\documentclass[10pt,a4paper,reqno]{amsart}
\usepackage{geometry}
\usepackage{amsmath}
\usepackage{amssymb}
\usepackage{amsthm}
\usepackage{xcolor}
\usepackage{verbatim}
\usepackage{tabularx}
\usepackage{mathrsfs}

\usepackage{hyperref}

\theoremstyle{theorem}
\newtheorem{theorem}{Theorem}[section]
\newtheorem{lemma}[theorem]{Lemma}
\newtheorem{proposition}[theorem]{Proposition}
\newtheorem{corollary}[theorem]{Corollary}

\theoremstyle{remark}
\newtheorem{remark}[theorem]{Remark}
\numberwithin{equation}{section}

\theoremstyle{definition}
\newtheorem{definition}{Definition}[section]

\DeclareMathOperator{\diver}{div}
\DeclareMathOperator{\Riem}{Riem}
\DeclareMathOperator{\Ricc}{Ric}
\DeclareMathOperator{\Sect}{Sect}
\DeclareMathOperator{\trace}{tr}

\newcommand{\R}{\mathbb{R}}
\newcommand{\N}{\mathbb{N}}

\newcommand{\sphere}{\mathbb{S}}

\newcommand{\eps}{\varepsilon}

\newcommand{\di}{\mathrm{d}}

\renewcommand{\div}{\diver}
\newcommand{\metric}{\langle\,,\,\rangle}
\renewcommand{\phi}{\varphi}
\renewcommand{\emptyset}{\varnothing}

\makeatletter
\newcommand*\owedge{\mathpalette\@owedge\relax}
\newcommand*\@owedge[1]{
	\mathbin{
		\ooalign{
			$#1\m@th\bigcirc$\cr
			\hidewidth$#1\m@th\wedge$\hidewidth\cr
		}
	}
}
\makeatother

\title[Tachibana-type theorems on complete manifolds]{Tachibana-type theorems on complete manifolds}

\author{Giulio Colombo}
\address{Dipartimento di Matematica ``F. Enriques", Universit\`a degli studi di Milano, Via Saldini 50, I-20133 Milano (Italy).}
\email{giulio.colombo@unimi.it, marco.mariani1@unimi.it, marco.rigoli55@gmail.com}

\author{Marco Mariani}

\author{Marco Rigoli}

\thanks{}

\date{\today}

\begin{document}

\maketitle

\begin{abstract}
	We prove that a compact Riemannian manifold of dimension $m \geq 3$ with harmonic curvature and $\lfloor\frac{m-1}{2}\rfloor$-positive curvature operator has constant sectional curvature, extending the classical Tachibana theorem for manifolds with positive curvature operator. The condition of $\lfloor\frac{m-1}{2}\rfloor$-positivity originates from recent work of Petersen and Wink, who proved a similar Tachibana-type theorem under the stronger condition that the manifold be Einstein. We show that the same rigidity property holds for complete manifolds assuming either parabolicity, an integral bound on the Weyl tensor or a stronger pointwise positive lower bound on the average of the first $\lfloor\frac{m-1}{2}\rfloor$ eigenvalues of the curvature operator. For $3$-manifolds, we show that positivity of the curvature operator can be relaxed to positivity of the Ricci tensor.
\end{abstract}

\bigskip



\bigskip

\section{Introduction}

%

By a classical theorem of S.-I. Tachibana \cite{tachi74}, any compact Riemannian manifold $(M,g)$ of dimension $m\geq 3$ with harmonic curvature tensor and positive curvature operator has constant (positive) sectional curvature, hence it is isometric to a quotient of a sphere of constant curvature. If the curvature operator is only non-negative, then the manifold is locally symmetric.


The Riemann curvature tensor is harmonic if and only if the Ricci tensor is a Codazzi tensor, as a consequence of the second Bianchi identity. In particular, Einstein manifolds (of dimension at least $3$) have harmonic curvature. If $m\geq 3$, then the Ricci tensor is Codazzi if and only if the scalar curvature is constant and the Cotton tensor is zero. If $m\geq 4$ then this is equivalent to having constant scalar curvature and harmonic Weyl curvature tensor (whereas in dimension $m=3$ the Weyl tensor vanishes for any Riemannian metric).

Recently, H. Tran, \cite{tran17}, has proved that a compact Riemannian manifold $M$ of dimension $m\geq4$ with harmonic Weyl curvature tensor $W$ and positive curvature operator is locally conformally flat, that is, $W\equiv 0$. 
As positive curvature operator implies positive Ricci curvature, the following classification theorem of M. H. Noronha \cite{nor93} shows that in this case $M$ is globally \emph{conformal} to a quotient of a standard sphere.


\begin{theorem}[\cite{nor93}, Theorem 1 and Proposition 4.2] \label{noronha}
	Let $M$ be a locally conformally flat, compact manifold with $\Ricc\geq0$ and dimension $m\geq 3$. Then the universal cover of $M$ is either
	\begin{itemize}
		\item [(i)] globally conformally equivalent to $\sphere^m$, or
		\item [(ii)] isometric to $\sphere^{m-1}\times\R$ or $\R^m$.
	\end{itemize}
	If $M$ is also locally symmetric, then its universal cover is isometric to either $\sphere^m$, $\sphere^{m-1}\times\R$ or $\R^m$ (that is, the conformal equivalence in \emph{(i)} can be strengthened to isometry).
\end{theorem}

Tran also observed that if the curvature operator is non-negative then the Weyl tensor is parallel and so, by a theorem of A. Derdzi\'nski and D. Roter, \cite[Theorem 2]{dr77}, the manifold is either locally conformally flat or locally symmetric. Both cases can occur, as shown by the simple example of $M = \mathbb{T}^2 \times \sphere^{m-2}$, that has harmonic curvature and non-negative curvature operator, and is locally symmetric but not locally conformally flat. 
In view of Noronha's classification theorem, Tran's theorem can be summarized as

\begin{theorem}[\cite{tran17,pw21}]
	Let $M$ be a compact manifold of dimension $m\geq 4$ with harmonic Weyl tensor and non-negative curvature operator. Then $M$ is either globally conformal to a quotient of $\sphere^m$ or locally symmetric. If the curvature operator is positive at some point, then the first case occurs.
\end{theorem}

In the recent works \cite{pw20,pw21}, P. Petersen and M. Wink came to consider the more general case where $M$ is a compact Riemannian manifold with $\lfloor\frac{m-1}{2}\rfloor$-positive curvature operator. We recall that the curvature operator is said to be $k$-positive (resp., $k$-non-negative) if the sum of its $k$ smallest eigenvalues, counted according to multiplicity, is positive (resp., non-negative). For the sake of brevity, we introduce the following notation: denoting by $\mathfrak R : \wedge^2 M \to \wedge^2 M$ the curvature operator of $M$, for every $x\in M$ we let
$$
	\lambda_1(x) \leq \lambda_2(x) \leq \cdots \leq \lambda_{\binom{m}{2}}(x)
$$
be the eigenvalues of $\mathfrak R_x : \wedge^2_x M \to \wedge^2_x M$ repeated according to multiplicities and for every integer $1 \leq k \leq \binom{m}{2}$ we set
$$
	\mathfrak R^{(k)} = \frac{1}{k} \sum_{\alpha=1}^k \lambda_\alpha \, .
$$

In \cite{pw20} Petersen and Wink proved that a compact Einstein manifold with $\mathfrak R^{(\lfloor\frac{m-1}{2}\rfloor)} > 0$ has constant sectional curvature, and more generally it is locally symmetric if $\mathfrak R^{(\lfloor\frac{m-1}{2}\rfloor)} \geq 0$. In \cite{pw21} they showed that a compact manifold with harmonic Weyl tensor and $\mathfrak R^{(\lfloor\frac{m-1}{2}\rfloor)} \geq 0$ is either globally conformal to a quotient of $\sphere^m$ or locally symmetric, and that the first possibility always occurs if $\mathfrak R^{(\lfloor\frac{m-1}{2}\rfloor)} > 0$. 

Their proofs, as well as those of Tachibana and Tran, are examples of applications of the Bochner technique, originated by S. Bochner and K. Yano, \cite{boch46,by53}. If $T$ is a harmonic algebraic curvature tensor on a Riemannian manifold $M$ (that is, $T$ satisfies the second Bianchi identity and has zero divergence -- we refer to subsection \ref{subsec_T} for this and every other relevant definition) 
then
\begin{equation} \label{T_intro}
	\frac{1}{2} \Delta|T|^2 = |\nabla T|^2 + \frac{1}{2} \langle \Gamma T,T \rangle
\end{equation}
where $\Delta$ is the Laplace-Beltrami operator of $M$, $\nabla$ is the Levi-Civita connection and $\Gamma = \Gamma_4$ is one of the family $\{\Gamma_q\}$ of self-adjoint endomorphism $\Gamma = \Gamma_q : T^0_q M \to T^0_q M$ defined by A. Lichnerowicz in \cite{lich61} on the bundle of $q$-covariant tensor fields, for any $q\in\N$. A key point in the application of the Bochner technique consists in establishing effective lower bounds on the quadratic term $\langle \Gamma T,T \rangle$ in order to apply some form of the maximum principle (or the divergence theorem, if $M$ is compact) to equation \eqref{T_intro}.

In \cite{tachi74}, Tachibana dealt with the case $T=\Riem$ and showed that a lower bound on $\langle \Gamma T,T \rangle$ is implied by a lower bound for the first eigenvalue $\lambda_1$ of the curvature operator $\mathfrak R$. 
In \cite{pw20}, Petersen and Wink identified new curvature conditions in terms of the partial traces $\mathfrak R^{(k)}$ that are effective for the application of the Bochner technique to harmonic $p$-forms. As a consequence of their analysis they proved that a lower bound on $\mathfrak R^{(\lfloor\frac{m-1}{2}\rfloor)}$ yields a lower bound on $\langle \Gamma T,T \rangle$ whenever $T$ is an algebraic curvature tensor whose Ricci contraction is a multiple of the metric. This happens if $T$ is the Weyl tensor of any Riemannian manifold (whose Ricci contraction is the zero tensor) or if $T$ is the Riemann tensor of an Einstein manifold (whose Ricci contraction, that is, the Ricci tensor of the manifold, is by definition a multiple of the metric).


\subsection*{Main results}

One of the original contributions of this paper is a refinement of Petersen and Wink's estimate. Namely, we prove that a lower bound on $\mathfrak R^{(\lfloor\frac{m-1}{2}\rfloor)}$ yields a lower bound on $\langle \Gamma T,T \rangle$ for any algebraic curvature tensor $T$, without additional structural assumptions (Theorem \ref{thm_m-1/2_R}). 
As a consequence, we obtain the full generalization of Tachibana's theorem under the sharpened positivity conditions on the curvature operator, relaxing the condition that $M$ be Einstein to the weaker assumption that the curvature tensor is harmonic.

\begin{theorem} \label{int_cpt_tach}
	Let $M$ be a compact Riemannian manifold of dimension $m\geq 3$ with $\mathfrak R^{(\lfloor\frac{m-1}{2}\rfloor)} \geq 0$ and harmonic curvature. Then $M$ is locally symmetric. If $\mathfrak R^{(\lfloor\frac{m-1}{2}\rfloor)} > 0$ at some point then $M$ is a quotient of $\sphere^m$.
\end{theorem}

In cases $m=3,4$ we have $\lfloor\frac{m-1}{2}\rfloor = 1$, so $\mathfrak R^{(\lfloor\frac{m-1}{2}\rfloor)} = \mathfrak R^{(1)} = \lambda_1$ and the positivity assumptions on $\mathfrak R$ reduce to the standard ones of Tachibana's theorem. In case $m=3$, in which harmonic curvature is equivalent to local conformal flatness and constant scalar curvature, see \cite[page 92]{eis49}, the lower bounds on the curvature operator can in fact be relaxed to lower bounds on the Ricci tensor. This may come unsurprising to the expert reader, since it is known from the literature, see R. Hamilton's \cite{ham82}, that in dimension $m=3$ the curvature terms in the Bochner identity for the Riemann tensor can be controlled assuming only $\Ricc\geq0$, although this condition is much less demanding than even non-negative sectional curvature (see \cite[Corollary 8.2]{ham82}). However, we are not aware of any reference in the literature to such a ``modified Tachibana's theorem'' for the $3$-dimensional case and thus we provide a self-contained proof (see Theorem \ref{thm_tachi_3} in Section \ref{sec_Tachi}).

\begin{theorem} \label{int_cpt_t3}
	Let $M^3$ be a compact $3$-dimensional Riemannian manifold with $\Ricc\geq0$. If $M$ is locally conformally flat with constant scalar curvature, then it is isometric to a quotient of $\sphere^3$, $\sphere^2\times\R$ or $\R^3$. Moreover, if $\Ricc>0$ at some point then $M$ is a quotient of $\sphere^3$.
\end{theorem}

We also deal with the complete case. A key point is the observation that condition $\mathfrak R^{(\lfloor\frac{m-1}{2}\rfloor)} \geq 0$ implies $\Ricc\geq0$, hence some powerful tools from the theory of complete Riemannian manifolds with nonnegative Ricci curvature are available. Again, the picture is a bit different in cases $m=3$ and $m\geq 4$. For complete $3$-manifolds we have the following Theorem \ref{int_cplt_t3}, whose statement is formally similar to that of Theorem \ref{int_cpt_t3} even though the proof relies on an additional classification theorem for locally conformally flat complete manifolds with $\Ricc\geq0$.

\begin{theorem} \label{int_cplt_t3}
	Let $M^3$ be a complete $3$-dimensional Riemannian manifold with $\Ricc\geq0$. If $M$ is locally conformally flat with constant scalar curvature, then it is isometric to a quotient of $\sphere^3$, $\sphere^2\times\R$ or $\R^3$. Moreover, if $\Ricc>0$ at some point then $M$ is a quotient of $\sphere^3$.
\end{theorem}

If $m\geq 4$ then we find ourselves bound to make some additional assumptions on the geometry of $M$. We present three different results (Theorems \ref{int_para_tach1}, \ref{int_cpl_tach1} and \ref{int_cpl_tach2}). In the first one, we assume slow volume growth of the manifold.

\begin{theorem} \label{int_para_tach1}
	Let $M$ be a complete, Riemannian manifold of dimension $m\geq 4$ with harmonic curvature and $\mathfrak R^{(\lfloor\frac{m-1}{2}\rfloor)} \geq 0$. Assume that for some fixed reference point $o\in M$
	\begin{equation} \label{parab}
		\limsup_{R\to+\infty} \int_1^R \frac{t}{|B_t|} \, \di t = +\infty
	\end{equation}
	where $|B_t|$ is the volume of the geodesic ball $B_t$ of radius $t$ with center at $o$. Then $M$ is locally symmetric. Moreover, if $\mathfrak R^{(\lfloor\frac{m-1}{2}\rfloor)} > 0$ somewhere then $M$ is isometric to a quotient of $\sphere^m$.
\end{theorem}

%
%


For any $m\geq 2$ the product $\R^2 \times \sphere^{m-2}$ provides an example of complete, non-compact manifold that satisfies all the assumptions in Theorem \ref{int_para_tach1} and is not locally conformally flat. On the contrary, the hypotheses of Theorems \ref{int_cpl_tach1} and \ref{int_cpl_tach2} below yield local conformal flatness of $M$.

%

\begin{theorem} \label{int_cpl_tach1}
	Let $M$ be a complete Riemannian manifold of dimension $m\geq 4$ with harmonic curvature and $\mathfrak R^{(\lfloor\frac{m-1}{2}\rfloor)} \geq 0$. If the Weyl tensor satisfies
	\begin{equation} \label{Wp_0}
		\lim_{R\to+\infty} \frac{1}{|B_R|} \int_{B_R} |W|^p = 0
	\end{equation}
	for some $p\in[1,+\infty)$, then $M$ is isometric to a quotient of $\sphere^m$, $\sphere^{m-1}\times\R$ or $\R^m$. Moreover, if $\mathfrak R^{(\lfloor\frac{m-1}{2}\rfloor)} > 0$ somewhere then $M$ is isometric to a quotient of $\sphere^m$.
\end{theorem}

%
%
%

\begin{theorem} \label{int_cpl_tach2}
	Let $M$ be a complete Riemannian manifold of dimension $m\geq 4$ with harmonic curvature. Assume that
	\begin{equation} \label{Riem_c_bound}
		\mathfrak R^{(\lfloor\frac{m-1}{2}\rfloor)} \geq 0 \quad \text{on } \, M \, , \qquad \qquad \mathfrak R^{(\lfloor\frac{m-1}{2}\rfloor)} \geq \frac{c}{1+r^2} \quad \text{on } \, M \setminus K
	\end{equation}
	for some compact set $K\subsetneq M$ and some $c>0$, where $r$ is the distance function from a fixed origin $o\in M$. Then $M$ is a quotient of $\sphere^m$.
\end{theorem}

A few remarks are in order. In the assumptions of Theorem \ref{int_cpl_tach1}, if $M$ is compact then condition \eqref{Wp_0} amounts to $W\equiv 0$ and Theorem \ref{int_cpt_tach} ensures that $M$ is locally symmetric, so the conclusion follows from Theorem \ref{noronha}. As a consequence of Theorem \ref{int_cpl_tach1} we have the following

\begin{corollary} \label{cor_cpl_tach1}
	Let $M$ be a complete Einstein manifold of dimension $m\geq 4$ with $\mathfrak R^{(\lfloor\frac{m-1}{2}\rfloor)} \geq 0$. If the Weyl tensor satisfies \eqref{Wp_0} then $M$ has constant sectional curvature.
\end{corollary}

\begin{remark}
	In \cite{cdh21}, G.~Cho, N.~T.~Dung and T.~Q.~Huy proved that a complete, noncompact Einstein manifold $M$ with $\mathfrak R^{(\lfloor\frac{m-1}{2}\rfloor)} \geq 0$ has constant sectional curvature provided $|W|\in L^p(M)$ for some $p\geq 2$. Corollary \ref{cor_cpl_tach1} shows that the summability condition on $|W|$ can be slightly relaxed. Indeed, under the aforementioned assumptions $M$ happens to be a complete, noncompact manifold with $\Ricc\geq0$, hence its volume is infinite (see for instance \cite[page 25]{sy94}) and if $|W|\in L^p(M)$ then \eqref{Wp_0} is satisfied.
\end{remark}

\begin{remark}
	The second condition in \eqref{Riem_c_bound} implies the following lower bound on the Ricci tensor
	\begin{equation} \label{Ric_c_bound}
		\Ricc \geq \frac{m-1}{2} \frac{c}{1+r^2} \qquad \text{on } \, M \setminus K \, .
	\end{equation}
	By a theorem of G. Galloway, \cite{gal82}, see also J. Cheeger, M. Gromov, M. Taylor, \cite[Theorem 4.8]{cgt82}, a sufficient condition for compactness of a complete Riemannian manifold $M^m$ is that
	$$
	\liminf_{s\to+\infty} s^2 \Ricc(\dot\gamma(s),\dot\gamma(s)) > \frac{m-1}{4}
	$$
	for every unit speed geodesic $\gamma : [0,+\infty) \to M$ issuing from a fixed origin $o\in M$, and the constant $\frac{m-1}{4}$ is sharp in this respect. Hence, if \eqref{Riem_c_bound} holds with $c > 1/2$ then compactness of $M$ is a priori guaranteed via \eqref{Ric_c_bound}. However, we allow $c > 0$ to be arbitrarily small in Theorem \ref{int_cpl_tach2}.
\end{remark}

\subsection*{Sketch of the proofs}

The proofs of Theorems \ref{int_cplt_t3}, \ref{int_para_tach1}, \ref{int_cpl_tach1} and \ref{int_cpl_tach2} bear some similarities. To illustrate the main circle of ideas, we sketch the argument of the proof of Theorem \ref{int_cpl_tach1}: 
since $|W|$ happens to be a subharmonic function (see formula \eqref{Delta_W} below) on a complete manifold with $\Ricc\geq0$, by Li-Schoen's mean value inequalities for non-negative subharmonic functions as a consequence of \eqref{Wp_0} we have $W\equiv 0$. Then we apply the following classification theorem for locally conformally flat \emph{complete} Riemannian manifolds with non-negative Ricci curvature, which is a refinement by G. Carron and M. Herzlich, \cite{ch06}, of a result of S. Zhu, \cite{zhu94}.

\begin{theorem}[\cite{ch06}] \label{carherz}
	Let $M$ be a locally conformally flat, complete Riemannian manifold with $\Ricc\geq0$ and dimension $m\geq 3$. Then the universal cover of $M$ is either
	\begin{itemize}
		\item [(i)] isometric to $\R^m$,
		\item [(ii)] isometric to $\sphere^{m-1}\times\R$,
		\item [(iii)] globally conformally equivalent to $\sphere^m$, or
		\item [(iv)] non-flat and globally conformally equivalent to $\R^m$.
	\end{itemize}
\end{theorem}

\noindent Cases (i) and (ii) already fit into the thesis of Theorem \ref{int_cpl_tach1}. In case (iii) -- where $M$ happens to be compact! -- we show that conformal equivalence can be strengthened to isometry, using Theorem \ref{int_cpt_tach} to deduce local symmetry of $M$ and then applying Noronha's Theorem \ref{noronha}. Lastly, alternative (iv) is ruled out by a contradiction argument drawn from the proof of Theorem 1.1 of \cite{prs07}: if (iv) were satisfied under the assumptions of our Theorem \ref{int_cpl_tach1}, then the universal cover $\tilde M$ of $M$ would be a complete Riemannian manifold of constant positive scalar curvature conformally equivalent to the Euclidean space $\R^m$, and this is impossible by the celebrated rigidity theorem of L. Caffarelli, B. Gidas and J. Spruck, \cite[Corollary 8.2]{cgs89}.

As just remarked, a key point in the proof of Theorem \ref{int_cpl_tach1} is the fact that harmonic curvature implies constant scalar curvature. In this work we also observe that for any integer $k < \binom{m}{2}$ the $k$-non-negativity of the curvature operator implies a pointwise pinching condition
\begin{equation} \label{S_pinch}
	|\Riem| \leq c_{m,k} S
\end{equation}
on the full norm $|\Riem|$ of the Riemann tensor in terms of the (necessarily non-negative) scalar curvature function $S$, see Corollary \ref{S_Riem_bound}. Hence, harmonic curvature and condition $\mathfrak R^{(\lfloor\frac{m-1}{2}\rfloor)} \geq 0$, together with a suitable Bochner inequality, imply that the function $|W| \leq |\Riem| \leq c_{m,k} S$ is a bounded solution of
\begin{equation} \label{Delta_W}
	\Delta |W| \geq (m-1)\mathfrak R^{(\lfloor\frac{m-1}{2}\rfloor)}|W| \, .
\end{equation}
If $M$ satisfies the slow volume growth condition \eqref{parab} then it is \emph{parabolic}, that is, the only upper bounded subharmonic functions on $M$ are the constant functions. In particular in this case $|W|$ is constant, and it must vanish if $\mathfrak R^{(\lfloor\frac{m-1}{2}\rfloor)}>0$ somewhere on $M$. If \eqref{parab} is not in force but $\mathfrak R^{(\lfloor\frac{m-1}{2}\rfloor)}$ satisfies a sufficiently strong pointwise lower bound as that in \eqref{Riem_c_bound}, then $|W|$ vanishes as well. These are the starting points in the proofs of Theorems \ref{int_para_tach1} and \ref{int_cpl_tach2}.

\subsection*{Generic algebraic curvature tensors}

In the compact setting, the argument of the proof of Theorem \ref{int_cpt_t3} carries on without modification for any harmonic algebraic curvature tensor. Hence, we have the following

\begin{theorem} \label{int_cpt_gen}
	Let $M^m$, $m\geq 3$, be a compact Riemannian manifold with $\mathfrak R^{(\lfloor\frac{m-1}{2}\rfloor)} \geq 0$. If $T$ is a harmonic algebraic curvature tensor on $M$, then $\nabla T \equiv 0$. Moreover, if $\mathfrak R^{(\lfloor\frac{m-1}{2}\rfloor)} > 0$ at some point then $T$ is a constant multiple of $\metric\owedge\metric$.
\end{theorem}

On complete manifolds the situation is a bit more complicated, as there seems to be no natural condition ensuring a priori boundedness for the norm of an arbitrary harmonic curvature tensor. However, if this additional assumption is made then we have analogues of Theorems \ref{int_para_tach1}, \ref{int_cpl_tach1} and \ref{int_cpl_tach2}. These are stated and proved in Section \ref{sec_Tachi} as Theorems \ref{gen_cpl_T1} and \ref{gen_cpl_T2}.

To exemplify the situation, let us consider the case of manifolds with harmonic Weyl curvature tensor. In this case we have no a priori constancy of the scalar curvature, so \eqref{S_pinch} does not immediately imply boundedness of $|W|$ but it certainly does if $S$ is just assumed to be bounded. 

\begin{theorem} \label{int_cpl_tran1}
	Let $M^m$ be a complete Riemannian manifold with $\mathfrak R^{(\lfloor\frac{m-1}{2}\rfloor)} \geq 0$ and harmonic Weyl tensor. Assume that
	$$
		\lim_{R\to+\infty} \frac{1}{|B_R|} \int_{B_R} |W|^p = 0
	$$
	for some $p\in[1,+\infty)$. Then $M$ is locally conformally flat.
\end{theorem}

\begin{theorem} \label{int_cpl_tran2}
	Let $M$ be a complete Riemannian manifold of dimension $m\geq 4$ with harmonic Weyl tensor and bounded scalar curvature. Assume that $\mathfrak R^{(\lfloor\frac{m-1}{2}\rfloor)} \geq 0$ and that either
	\begin{itemize}
		\item [(a)] $\mathfrak R^{(\lfloor\frac{m-1}{2}\rfloor)} > 0$ at some point and \eqref{parab} is satisfied, or
		\item [(b)] $\mathfrak R^{(\lfloor\frac{m-1}{2}\rfloor)}$ satisfies \eqref{Riem_c_bound} for some compact set $K\subsetneq M$, some constant $c>0$ and some fixed origin $o\in M$.
	\end{itemize}
	Then $M$ is globally conformally equivalent to a quotient of $\sphere^m$ or $\R^m$.
\end{theorem}

\begin{remark}
	In the assumptions of both Theorems \ref{int_cpl_tran1} and \ref{int_cpl_tran2}, $M$ happens to be a complete locally conformally flat manifold with $\Ricc\geq0$. Alternatives (i), (ii) and (iii) in Theorem \ref{carherz} are all compatible with the condition $\mathfrak R^{(\lfloor\frac{m-1}{2}\rfloor)} \geq 0$ (the manifolds in (i) and (ii) have $\mathfrak R\geq 0$ and, by continuity, $\sphere^m$ can be conformally deformed to manifold of non-constant sectional curvature with $\mathfrak R > 0$), but (i) and (ii) are incompatible with the strict inequality $\mathfrak R^{(\lfloor\frac{m-1}{2}\rfloor)} > 0$ at any point and therefore are excluded from the conclusion of Theorem \ref{int_cpl_tran2}. It is natural to ask if case (iv) in Theorem \ref{carherz} is compatible with the condition $\mathfrak R^{(\lfloor\frac{m-1}{2}\rfloor)} \geq 0$. Note that in the sketch of the proof of Theorem \ref{int_cpl_tach1} case (iv) was dismissed as incompatible with constant scalar curvature, but this is not the case in Theorems \ref{int_cpl_tran1} and \ref{int_cpl_tran2}.
\end{remark}

\begin{remark}
	We stress that in the terminology adopted here, an algebraic curvature tensor $T$ is said to be harmonic if it has zero divergence and satisfies the second Bianchi identity, that is,
	\begin{equation} \label{int_2B}
		(\nabla_V T)(X,Y,Z,W) + (\nabla_W T)(X,Y,V,Z) + (\nabla_V T)(X,Y,W,Z) = 0
	\end{equation}
	for all vector fields $X,Y,Z,W,V$. These combined conditions amount to $\Delta_L T = 0$ (thus justifying the term ``harmonic''), where $\Delta_L$ is the Lichnerowicz Laplacian given by $\Delta_L = \Delta_B + \frac{1}{2}\Gamma$ with $\Delta_B = -\div(\nabla\,\cdot\,)$ the Bochner (or rough) Laplacian on covariant tensors and $\Gamma$ the operator mentioned above.
	
	The Riemann curvature tensor always satisfies the second Bianchi identity, so it is harmonic if and only if it has zero divergence. For the Weyl curvature tensor $W$ the validity of the second Bianchi identity is equivalent to condition $\div W = 0$, so it turns out that $\Delta_L W = 0$ is equivalent to $\div W = 0$. 
	However, for an arbitrary algebraic curvature tensor $T$ the two conditions $\div T = 0$ and \eqref{int_2B} are generally independent from each other, so in particular $\div T = 0$ is not necessarily equivalent to harmonicity of $T$ in the present sense. Some examples of this fact can be observed in geometrically relevant situations:
	\begin{itemize}
		\item [(i)] If $(M,g)$ is a Ricci soliton with potential $f\in C^\infty(M)$, that is, there exists $\lambda\in\R$ such that
		$$
			\Ricc + \mathrm{Hess} f = \lambda g
		$$
		then the algebraic curvature tensor $e^{-f}\Riem$ satisfies $\div(e^{-f}\Riem) = 0$ as this is equivalent, more generally, to the condition that $\Ricc + \mathrm{Hess} f$ is a Codazzi tensor. However, in general $e^{-f}\Riem$ does not satisfy the second Bianchi identity.
		\item [(ii)] If $\varphi : (M,g_M) \to (N,g_N)$ is a harmonic map between Riemannian manifolds such that
		\begin{equation} \label{hE_ex}
			\Ricc_M = \lambda g_M + \alpha \varphi^\ast g_N
		\end{equation}
		for some $\lambda\in\R$ and some coupling constant $\alpha\in\R^\ast$, then the structure $(M,N,\varphi)$ is said to be a harmonic Einstein structure, see \cite{wang16}. Harmonic Einstein structures are fixed points of the coupled Ricci-harmonic map flow introduced by R. Buzano, \cite{mull12}. 
		In \cite{acr21}, in a more general context it has been introduced an algebraic curvature tensor, the $\varphi$-Weyl tensor $W^\varphi$, which reflects the part of the Riemann tensor of $M$ that is not prescribed by the algebraic structure of $\Ricc^\varphi = \Ricc_M - \alpha \varphi^\ast g_N$. For a harmonic Einstein structure the tensor $W^\varphi$ satisfies the second Bianchi identity, see \cite[Proposition 3.2]{mr20}, but in general it is not divergence free (unless the pull-back metric $\varphi^\ast g_N$ is a Codazzi tensor on $M$).
	\end{itemize}
\end{remark}

\subsection*{Plan of the paper}

In Section \ref{sec_prel} we fix our notation and terminology, we review the definition and relevant properties of Lichnerowicz' operators $\Gamma$ and we collect a series of facts about algebraic curvature tensors. In Section \ref{sec_Boch} we prove the Bochner identity for smooth algebraic curvature tensors and we establish the lower bound on $\langle \Gamma T,T\rangle$ in term of $\mathfrak R^{(\lfloor\frac{m-1}{2}\rfloor)}$ for an arbitrary algebraic curvature tensor $T$. In Section \ref{sec_Tachi} we apply this to the proof of Tachibana-type theorems both in the compact setting and in the complete one. Theorems \ref{int_cpt_tach}, \ref{int_cpt_t3}, \ref{int_cplt_t3}, \ref{int_para_tach1}, \ref{int_cpl_tach1}, \ref{int_cpl_tach2}, \ref{int_cpt_gen}, \ref{int_cpl_tran1}, \ref{int_cpl_tran2} from this introduction correspond to Theorems \ref{cor_cpt_tachi}, \ref{thm_tachi_3}, \ref{cpl_tachi_dim3}, \ref{para_cplt_tachi}, \ref{cplt_tachi_1}, \ref{cplt_tachi_2}, \ref{cpt_Tachi}, \ref{gen_cpl_T1}, \ref{gen_cpl_T2} there.

\subsection*{Acknowledgements}

The authors are thankful to Luciano Mari for suggestions on preliminary drafts of this note, that led to improvements in the overall presentation of the results and in particular to the formulation of Theorem \ref{int_para_tach1}.

\section{Preliminaries} \label{sec_prel}

\subsection{Notation}

Let $(M,\metric)$ be a Riemannian manifold of dimension $m$. For each positive integer $q$, the Riemannian metric $\metric$ induces an inner product, that we still denote with $\metric$, on the bundle $T^0_q M$. The standard construction is the following: for any $x\in M$ and $\alpha,\beta\in T^\ast_x M$ we set
$$
	\langle \alpha,\beta \rangle = \langle a,b \rangle
$$
with $a,b\in T_x M$ the vectors metrically equivalent to $\alpha$ and $\beta$, respectively, then we extend $\metric$ to $T^0_{q,x} M \times T^0_{q,x} M = (T^\ast_x M)^{\otimes q} \times (T^\ast_x M)^{\otimes q}$ by $q$-linearity in both variables. We also set
$$
	|A| = \sqrt{\langle A,A \rangle} \qquad \forall \, A \in T^0_{q,x} M \, , x \in M .
$$


In the following, we will perform many computations in local notation. Let $\{e_i\}_{1\leq i\leq m}$ be a local reference frame for $TM$ defined on an open subset $U\subseteq M$ and let $\{\theta^i\}_{1\leq i\leq m}$ be its dual coframe, acting as a local reference frame for $T^\ast M$ on $U$. The Riemannian metric is written as
$$
	\metric = g_{ij} \, \theta^i \otimes \theta^j \qquad \text{on } \, U
$$
where we adopt the Einstein summation convention over repeated indices. Letting $(g^{ij}) = (g_{ij})^{-1}$ as matrices, we have $g_{ij} = \langle e_i,e_j \rangle$ and $g^{ij} = \langle \theta^i,\theta^j \rangle$ for any $1\leq i,j\leq m$. Hence, for any $q\in\N$ we can describe the inner product on $T^0_q M$ as follows: for every pair of $(0,q)$-type tensor fields
$$
	A = A_{i_1\cdots i_p} \, \theta^{i_1} \otimes \cdots \otimes \theta^{i_q} \, , \qquad B = B_{i_1\cdots i_q} \, \theta^{i_1} \otimes \cdots \otimes \theta^{i_q}
$$
we have
$$
	\langle A,B \rangle = g^{i_1 j_1} \cdots g^{i_q j_q} A_{i_1\dots i_q} B_{j_1 \dots j_q} = A_{i_1 \dots i_q} B^{i_1 \dots i_q}
$$
where we are also adopting the convention of lowering or raising indexes to denote contraction with $g$ or $g^{-1}$, respectively. Note that when the local reference frame $\{e_i\}$ is chosen to be orthonormal we have $g_{ij} = g^{ij} = \delta_{ij}$ for $1\leq i,j\leq m$, with $\delta$ the Kronecker symbol.

If $A$ is a differentiable section of $T^0_q M$, we locally express its covariant derivative $\nabla A \in \Gamma(T^0_{q+1} M)$ and the iterations $\nabla^2 A = \nabla(\nabla A) \in \Gamma(T^0_{q+2} M)$, $\dots$ as
\begin{align*}
	\nabla A & = A_{i_1\cdots i_p,j} \, \theta^j \otimes \theta^{i_1} \otimes \cdots \otimes \theta^{i_q} \, , \\
	\nabla^2 A & = A_{i_1\cdots i_p,jk} \, \theta^k \otimes \theta^j \otimes \theta^{i_1} \otimes \cdots \otimes \theta^{i_q} \, ,
\end{align*}
and so on. The divergence of $A$ is the tensor field $\div A$ of type $(0,q-1)$ defined by
$$
	(\div A)(X_1,\dots, X_{q-1}) = \trace_g [ (Y,Z) \mapsto (\nabla_Y A)(Z,X_1,\dots,X_{q-1}) ] \, .
$$
for every $X_1,\dots, X_{q-1} \in \mathfrak X(M)$, and with our notation for $\nabla A$ we have
$$
	(\div A)_{i_1\dots i_{q-1}} = A^j_{\;i_1\dots i_{q-1},j} \, .
$$

\begin{remark}
	It is also customary to locally denote covariant differentiation by placing a subscript on the symbol $\nabla$, so that alternative notations for the coefficients of $\nabla A$, $\nabla^2 A$, $\dots$ are
	$$
		A_{i_1\cdots i_p,j} = \nabla_j A_{i_1\cdots i_p} \, , \qquad A_{i_1\cdots i_p,jk} = \nabla_k \nabla_j A_{i_1\cdots i_p} \, , \qquad \dots
	$$
	and we have
	$$
		(\div A)_{i_1\dots i_{q-1}} = \nabla^j A_{j i_1\dots i_{q-1}} \, .
	$$
	However, we shall not use this notation in the sequel.
\end{remark}

\subsection{Curvature operator}

We define the Riemann curvature tensor $\Riem$ by setting
$$
	\Riem(W,Z,X,Y) = \langle W,R(X,Y)Z \rangle = \langle W,\nabla_X\nabla_Y Z - \nabla_Y\nabla_X Z - \nabla_{[X,Y]}Z \rangle
$$
for every $X,Y,Z,W\in\mathfrak X(M)$, so that the Ricci tensor is given by the Ricci contraction
$$
	\Ricc(X,Y) = \trace_g [ (Z,W) \mapsto \Riem(Z,X,W,Y) ]
$$
for every $X,Y\in\mathfrak X(M)$ and the scalar curvature is the trace
$$
	S = \trace_g \Ricc \, .
$$
With respect to a local coframe $\{\theta^i\}$ we write
$$
	\Riem = R_{ijkt} \, \theta^i \otimes \theta^j \otimes \theta^k \otimes \theta^t \, , \qquad \Ricc = R_{ij} \, \theta^i \otimes \theta^j
$$
and we have $R_{ij} = R^k_{\;ikj}$, $S = R^i_{\;i} = R^{ij}_{\;\;ij}$. The Riemann tensor has the symmetries
\begin{alignat*}{2}
	& R_{ijkt} = - R_{jikt} = R_{ktij} && \qquad \forall \, 1 \leq i,j,k,t \leq m \\
	\intertext{and satisfies the first and second Bianchi identities}
	& R_{ijkt} + R_{itjk} + R_{iktj} = 0 && \qquad \forall \, 1 \leq i,j,k,t \leq m \\
	& R_{ijkt,l} + R_{ijlk,t} + R_{ijtl,k} = 0 && \qquad \forall \, 1 \leq i,j,k,t,l \leq m \, .
\end{alignat*}

The symmetries of $\Riem$ allow us to define a linear, self-adjoint endomorphism $\mathfrak R$, the curvature operator, on the space $\wedge^2 M$ of $2$-forms on $M$. With respect to a local coframe $\{\theta^i\}$, for every $2$-form $\omega = \omega_{ij} \, \theta^i \otimes \theta^j \equiv \frac{1}{2} \omega_{ij} \, \theta^i \wedge \theta^j$ we let $\mathfrak R\omega = (\mathfrak R\omega)_{kt} \, \theta^k \otimes \theta^t$ be given by
\begin{equation} \label{R_op_def}
	(\mathfrak R\omega)_{kt} = R_{ijkt} \omega^{ij} \, .
\end{equation}
For every $x\in M$, we denote by $\{\lambda_\alpha(x)\}_{1\leq\alpha\leq\binom{m}{2}}$ the non-decreasing sequence of the eigenvalues of $\mathfrak R_x : \wedge^2_x M \to \wedge^2_x M$ repeated according to multiplicity. We also let $\{\omega^\alpha\}_\alpha$ be an orthonormal basis for $\wedge^2_x(M)$ consisting of eigenvectors of $\mathfrak R$ corresponding to $\{\lambda_\alpha\}_\alpha$. Then, in local notation
\begin{equation} \label{R_spec}
	R_{ijkt} = \sum_\alpha \lambda_\alpha \omega^\alpha_{ij} \omega^\alpha_{kt} \, , \qquad \frac{1}{2} (g_{ik} g_{jt} - g_{it} g_{jk}) = \sum_\alpha \omega^\alpha_{ij} \omega^\alpha_{kt} \, .
\end{equation}

\begin{definition}
	Let $M$ be a Riemannian manifold of dimension $m\geq 2$. For $k \in \{1,\dots,\binom{m}{2}\}$, the $k$-th (normalized) partial trace of $\mathfrak R$ is the function
	\begin{equation} \label{R(k)_def}
		x \mapsto \mathfrak R^{(k)}(x) = \inf_{\substack{V\leq \wedge^2_x M \\ \dim V = k}} \left( \frac{1}{k} \sum_{\alpha=1}^k \langle \mathfrak R\psi^\alpha,\psi^\alpha \rangle \right)
	\end{equation}
	where $\{\psi^\alpha\}_{\alpha=1}^k$ is any orthonormal basis of $V$.
\end{definition}

By standard linear algebra we have that the infimum in the RHS of \eqref{R(k)_def} is attained when $V = \mathrm{span}\{\omega^1,\dots,\omega^k\}$, so that
$$
	\mathfrak R^{(k)} = \frac{1}{k} \sum_{\alpha=1}^k \lambda_\alpha
$$
for every $k \in \{1,\dots,\binom{m}{2}\}$. In particular we observe that
\begin{equation} \label{R(h)(k)}
	\mathfrak R^{(h)} \geq \mathfrak R^{(k)} \qquad \text{for every } \, 1 \leq k \leq h \leq \binom{m}{2}
\end{equation}
as a consequence of the following elementary

\begin{lemma} \label{lem_average}
	Let $N\geq1$ and let $\{a_i\}_{1\leq i\leq N}$ be a nondecreasing sequence of real numbers. Then
	$$
	\frac{1}{h} \sum_{i=1}^h a_i \geq \frac{1}{k} \sum_{i=1}^k a_i \qquad \text{for every } \, 1 \leq k \leq h \leq N \, .
	$$
\end{lemma}

\begin{proof}
	By induction, it suffices to prove the inequality in case $k < N$ and $h = k+1$. Since $a_{k+1} \geq a_i$ for $1 \leq i \leq k$, we have $a_{k+1} \geq \frac{1}{k} \sum_{i=1}^k a_i$ and then
	$$
	\frac{1}{k+1} \sum_{i=1}^{k+1} a_i = \frac{k}{k+1} \frac{1}{k} \sum_{i=1}^k a_i + \frac{1}{k+1} a_{k+1} \geq \frac{k}{k+1} \frac{1}{k} \sum_{i=1}^k a_i + \frac{1}{k+1} \frac{1}{k} \sum_{i=1}^k a_i = \frac{1}{k} \sum_{i=1}^k a_i \, .
	$$
\end{proof}

A $2$-form $\omega$ is said to be decomposable if there exist $1$-forms $v,u$ such that $\omega = \frac{1}{2} v \wedge u$. In local components, this means that $\omega_{ij}$ can be expressed as
$$
	\omega_{ij} = \frac{1}{2} ( v_i u_j - u_j v_i ) \, .
$$
The values assumed by the quadratic form $\langle \mathfrak R \,\cdot\, , \,\cdot\, \rangle$ on decomposable $2$-forms are related to the sectional curvatures of $M$ up to normalization. For any $x\in M$ and for any $2$-plane $\pi\leq T_x M$, the sectional curvature $\Sect(\pi)$ of $\pi$ is given by
$$
	\Sect(\pi) = \frac{\Riem(X,Y,X,Y)}{|X|^2|Y|^2 - \langle X,Y \rangle^2}
$$
for any couple of tangent vectors $X,Y\in T_x M$ such that $\pi = \mathrm{span}\{X,Y\}$. The value of the quotient appearing on the right-hand side does not depend on the choice of the basis $\{X,Y\}$. If $u,v$ are the $1$-forms metrically equivalent to $X,Y$, respectively, and $\omega = \frac{1}{2} u \wedge v$, then
\begin{align*}
	R(X,Y,X,Y) & = R_{ijkt} u^i v^j u^k v^t = R_{ijkt} \omega^{ij} \omega^{kt} \\
	|X|^2 |Y|^2 - \langle X,Y\rangle^2 & = u_i u^i v_j v^j - u_i v^i u^j v_j = 2 \omega_{ij} \omega^{ij}
\end{align*}
so we have
\begin{equation} \label{Riem_K_basic}
	\Sect(\pi) = \frac{1}{2} \frac{\langle \mathfrak R \omega,\omega \rangle}{|\omega|^2} \, .
\end{equation}
In particular, a lower bound on $\mathfrak R^{(k)}$ yields a lower bound on the average of the sectional curvature on any collection of $k$ mutually orthogonal $2$-planes, and therefore also on the Ricci curvature if $k\leq m-1$. More precisely, let use give the following definitions.

\begin{definition}
	Let $M$ be a Riemannian manifold of dimension $m\geq 2$ and let $x\in M$. We say that two $2$-planes $\pi_1,\pi_2 \leq T_x M$ are mutually orthogonal, and we write $\langle\pi_1,\pi_2\rangle = 0$, if for some (equivalently, for any) choice of bases $\{X_1,Y_1\}$ and $\{X_2,Y_2\}$ of $\pi_1$ and $\pi_2$, respectively, the $2$-forms
	$$
		\omega_1 = \frac{1}{2} u_1 \wedge v_1 \qquad \text{and} \qquad \omega_2 = \frac{1}{2} u_2 \wedge v_2
	$$
	are orthogonal with respect to the inner product on $\wedge^2_x M$, where $u_1,u_2,v_1,v_2$ are the $1$-forms metrically equivalent to $X_1,X_2,Y_1,Y_2$, respectively.
\end{definition}

In particular, any two $2$-planes $\pi_1,\pi_2\leq T_x M$ are mutually orthogonal if either
\begin{itemize}
	\item [(i)] each one of them is contained in the orthogonal complement of the other, or
	\item [(ii)] $\dim(\pi_1\cap\pi_2) = 1$ and there exist three mutually orthogonal vectors $X,Y,Z\in T_x M$ such that $\pi_1 = \mathrm{span}\{X,Y\}$ and $\pi_2 = \mathrm{span}\{X,Z\}$.
\end{itemize}

\begin{definition} \label{def_Scalk}
	Let $M$ be a Riemannian manifold of dimension $m\geq 2$. For $k\in\{1,\dots,\binom{m}{2}\}$, the $k$-th averaged lower bound on the sectional curvature is the function
	$$
		x \mapsto \Sect^{(k)}(x) = \inf_{\{\pi_1,\dots,\pi_k\}} \left( \frac{1}{k} \sum_{i=1}^k \Sect(\pi_i) \right)
	$$
	where the infimum is taken with respect to $\{\pi_1,\dots,\pi_k\}$ varying among all collections of $k$ mutually orthogonal $2$-planes in $T_x M$.
\end{definition}

From the above definitions together with \eqref{Riem_K_basic} and a further application of Lemma \ref{lem_average}
\begin{equation} \label{Riem_K_bound}
	\Sect^{(h)} \geq \Sect^{(k)} \geq \frac{1}{2} \mathfrak R^{(k)} \qquad \forall \, 1 \leq k \leq h \leq \binom{m}{2} \, .
\end{equation}
In particular, for the (non-normalized) Ricci tensor we have
$$
	\Ricc \geq (m-1) \Sect^{(m-1)}
$$
and therefore
\begin{equation} \label{Riem_Ric_bound}
	\Ricc \geq (m-1) \Sect^{(k)} \geq \frac{m-1}{2} \mathfrak R^{(k)} \qquad \text{for any } \, 1 \leq k \leq m-1 \, .
\end{equation}

We conclude this subsection by showing that non-negativity of $\mathfrak R^{(k)}$ for some $k < \binom{m}{2}$ implies an upper bound on $|\Riem|$ in terms of the scalar curvature $S$. To this aim, we first observe that $|\Riem|$ and $S$ are equal, respectively, to the Hilbert-Schmidt norm and the trace of $\mathfrak R$, that is,
\begin{equation}
	|\Riem|^2 = \sum_\alpha \lambda_\alpha^2 \, , \qquad S = \sum_\alpha \lambda_\alpha \, .
\end{equation}
This can be directly seen from \eqref{R_spec}. Then, we apply the following

\begin{lemma}
	Let $N\geq 1$ and let $\{a_i\}_{1\leq i\leq N}$ be a nondecreasing sequence of real numbers. If
	\begin{equation} \label{k_nneg}
		\sum_{i=1}^k a_i \geq 0
	\end{equation}
	for some $k \in \{1,\dots,N-1\}$, then
	$$
		\sum_{i=1}^N a_i \geq \frac{1}{k} \left( \frac{1}{N} \sum_{i=1}^N a_i^2 \right)^{1/2} \, .
	$$
\end{lemma}

\begin{proof}
	We can find $j$ such that $|a_j| \geq \sqrt{\frac{1}{N}\sum_{i=1}^N a_i^2}$. By \eqref{k_nneg}, there exists $h\in\{1,\dots,k\}$ such that $a_i<0$ if $i < h$ and $a_i\geq0$ if $i \geq h$. Note that $a_i \geq 0$ for $i\geq k$. If $j \geq h$ then $a_j = |a_j|$, hence
	$$
		\sum_{i=1}^N a_i = \sum_{i=1}^k a_i + \sum_{i=k+1}^N a_i \geq \sum_{i=1}^k a_i + a_\ell \geq a_\ell \geq a_j = |a_j|
	$$
	for $\ell = \max\{j,k+1\}$. If $j < h$ then we observe that
	$$
		(k - h + 1) a_k \geq \sum_{i=h}^k a_i \geq - \sum_{i=1}^{h-1} a_i = \sum_{i=1}^{h-1} |a_i| \geq |a_j|
	$$
	where the second inequality is a rewriting of \eqref{k_nneg}, so
	$$
		\sum_{i=1}^N a_i = \sum_{i=1}^k a_i + \sum_{i=k+1}^N a_i \geq \sum_{i=k+1}^N a_i \geq (N-k) a_k \geq \frac{N-k}{k-h+1} |a_j| \, .
	$$
	In conclusion,
	$$
		\sum_{i=1}^N a_i \geq \min\left\{ 1, \frac{N-k}{k-h+1} \right\} |a_j| \geq \min\left\{ 1, \frac{N-k}{k-h+1} \right\} \left( \frac{1}{N} \sum_{i=1}^N a_i^2 \right)^{1/2}
	$$
	and, since $1 \leq k < N$ and $1 \leq k - h + 1 \leq k$, we have $\min\left\{1,\frac{N-k}{k-h+1}\right\} \geq \frac{1}{k}$.
\end{proof}

\begin{corollary} \label{S_Riem_bound}
	Let $M$ be a Riemannian manifold of dimension $m\geq 2$. If $\mathfrak R^{(k)} \geq 0$ for some $1 \leq k < \binom{m}{2}$ then $k^2\binom{m}{2} S^2 \geq |\Riem|^2$.
\end{corollary}

\subsection{The operator $\Gamma$ and its relation with $\mathfrak R$}


In \cite{lich61}, A. Lichnerowicz defined for every $q\geq 1$ a self-adjoint endomorphism $\Gamma = \Gamma_q : T^0_q M \to T^0_q M$ whose action can be described in the following way: in any local coframe $\{\theta^i\}$, for every $q$-covariant tensor $Q = Q_{i_1\dots i_q} \, \theta^{i_1} \otimes \cdots \otimes \theta^{i_q}$ the components of $\Gamma Q = (\Gamma Q)_{i_1\dots i_q} \, \theta^{i_1} \otimes \cdots \otimes \theta^{i_q}$ are given by
\begin{equation} \label{Gam_def}
	(\Gamma Q)_{i_1\dots i_q} = \sum_{l=1}^q R_{i_l j} Q^{\phantom{i_1\dots}j}_{i_1\dots\;\dots i_q} - \sum_{1\leq l\neq h\leq q} R_{i_l j i_h t} Q^{\phantom{i_1\dots}j\phantom{\dots}\!t}_{i_1\dots\;\dots\;\dots i_q}
\end{equation}
where on the right-hand side $j$ and $t$ occupy the $l$-th and $h$-th places, respectively, among the indexes of $Q$ (note that in the second term we do not necessarily have $l<h$). If $Q$ is twice continuously differentiable, then
\begin{equation} \label{Gam_Ric}
	(\Gamma Q)_{i_1\dots i_q} = \sum_{h=1}^q \left( Q^{\phantom{i_1\dots}t}_{i_1\dots\;\dots i_q,i_h t} - Q^{\phantom{i_1\dots}t}_{i_1\dots\;\dots i_q,t i_h} \right) \, .
\end{equation}
Indeed, considering Ricci identities
$$
	Q_{i_1\dots i_q,st} - Q_{i_1\dots i_q,ts} = - \sum_{l=1}^q R_{i_l jst} Q^{\phantom{i_1\dots}j}_{i_1\dots\;\dots i_q}
$$
we have, for $h = 1,\dots,q$,
$$
	Q_{i_1\dots k\dots i_q,i_h t} - Q_{i_1\dots k\dots i_q,t i_h} = - R_{kj i_h t} Q^{\phantom{i_1\dots i_{h-1}}\!j}_{i_1\dots i_{h-1}\;i_{h+1}\dots i_q} - \sum_{h\neq l=1}^q R_{i_l j i_h t} Q^{\phantom{i_1\dots}j}_{i_1\dots\;\dots k\dots i_q} \, .
$$
Hence summing over $1 \leq h \leq q$ and tracing with respect to $k$ and $t$ yields \eqref{Gam_Ric}.

Establishing lower bounds on the quadratic form $\langle \Gamma \,\cdot\, , \,\cdot\, \rangle$ acting on (subbundles of) $T^0_q M$ is a key point in the development of the Bochner technique. The restriction of $\langle \Gamma \,\cdot\, , \,\cdot\, \rangle$ to the subbundle $\wedge^q M \subseteq T^0_q M$ was already considered in S. Bochner and K. Yano's book \cite[Formula (3.6)]{by53}. In case $q=1$ it is known that a lower bound on $\Ricc$ is sufficient to establish a lower bound on $\langle \Gamma\,\cdot\,,\,\cdot\, \rangle$, and this is the original key idea of Bochner in \cite{boch46}. In case $q\geq 2$, M. Berger \cite{ber61} and then D. Meyer \cite{mey71} showed that a lower bound on $\mathfrak R$ is sufficient, and generally needed, to give a lower bound on $\langle \Gamma\,\cdot\,,\,\cdot\, \rangle$ on $\wedge^q M$. Remarkably, Berger and D. Ebin observed in \cite[Proposition 6.1]{be69} that a lower bound on the sectional curvature of $M$ is enough to ensure a lower bound on $\langle \Gamma E,E \rangle$ when $E$ is a symmetric bilinear form.

The curvature operator $\mathfrak R$ naturally extends to a self-adjoint endomorphism
$$
	\mathfrak R^{T^0_q M} : T^0_q M \otimes \wedge^2 M \to T^0_q M \otimes \wedge^2 M
$$
on the bundle $T^0_q M \otimes \wedge^2 M$ of $T^0_q M$-valued $2$-forms, where self-adjointness is intended with respect to the inner product on $T^0_{q+2} M \supseteq T^0_q M \otimes \wedge^2 M$. Given a local coframe $\{\theta^i\}$ on $M$, for any section $\omega = \omega_{i_1\dots i_q sr} \, \theta^{i_1} \otimes \cdots \otimes \theta^{i_q} \otimes \theta^s \otimes \theta^r$ of $T^0_q M \otimes \wedge^2 M$ the tensor $\mathfrak R^{T^0_q M}\omega$ is locally defined by
$$
	(\mathfrak R^{T^0_q M} \omega)_{i_1\dots i_q kt} = R_{srkt} \omega^{\phantom{i_1\dots i_q}sr}_{i_1\dots i_q} \, .
$$
Following an idea of Berger \cite{ber61} later clarified by Meyer \cite{mey71}, and adopting the notation used by Petersen and Wink in \cite{pw20}, to any tensor $Q$ of type $(0,q)$ we associate a $T^0_q M$-valued $2$-form $\hat Q$ of local components
\begin{equation} \label{Qhat}
	\hat Q_{i_1\dots i_q sr} = \frac{1}{2}\sum_{l=1}^q Q_{i_1 \dots s \dots i_q} g_{i_l r} - \frac{1}{2}\sum_{l=1}^q Q_{i_1 \dots r \dots i_q} g_{i_l s} \, .
\end{equation}
Any symmetry that $Q$ may enjoy is inherited by $\hat Q$ in its first $q$ indexes. We have
\begin{align*}
	(\mathfrak R^{T^0_q M} \hat Q)_{i_1\dots i_q kt} = \sum_{l=1}^q R_{si_l kt} Q^{\phantom{i_1\dots}s}_{i_1 \dots \; \dots i_q} = - \sum_{l=1}^q R_{i_l j kt} Q^{\phantom{i_1\dots}j}_{i_1 \dots \; \dots i_q}
\end{align*}
and if $V$ is another tensor field of type $(0,q)$ then
$$
	\langle \mathfrak R^{T^0_q M} \hat Q, \hat V \rangle = - \sum_{l,h=1}^q R_{i_l jkt} Q^{\phantom{i_1\dots}j}_{i_1 \dots \; \dots i_q} V^{i_1\dots k \dots i_q} g^{i_h t}
$$
where on the right-hand side $j$ occupies the $l$-th place among the indexes of $Q$, and $k$ occupies the $h$-th place among the indexes of $V$. Splitting the cases $h = l$ and $h\neq l$, we get
$$
	\langle \mathfrak R^{T^0_q M} \hat Q,\hat V \rangle = \sum_{l=1}^q R_{jk} Q^{\phantom{i_1\dots}j}_{i_1 \dots \; \dots i_q} V^{i_1\dots k \dots i_q} - \sum_{1 \leq l\neq h\leq q} R_{i_l jkt} Q^{\phantom{i_1\dots}j\phantom{\dots}\!t}_{i_1 \dots \; \dots \; \dots i_q} V^{i_1\dots k \dots i_q}
$$
and renaming $k = i_h$ it is apparent that
\begin{equation} \label{Gam_hat}
	\langle \mathfrak R^{T^0_q M} \hat Q,\hat V \rangle = \langle \Gamma Q,V \rangle \, .
\end{equation}
This shows that a lower bound on the quadratic form $\langle \mathfrak R \,\cdot\, , \,\cdot\, \rangle$ implies a lower bound on $\langle \Gamma \,\cdot\, , \,\cdot\, \rangle$. Indeed, if $\{\theta^i\}$ is orthonormal and for any $q$-uple $(i_1,\dots,i_q) \in \{1,\dots,m\}^q$ we define the $2$-form
$$
	Q^{(i_1,\dots,i_q)} = \hat Q_{i_1\dots i_q sr} \, \theta^s \otimes \theta^r
$$
(this is more similar to the approach also used by Tachibana in \cite{tachi74}) then
$$
	\langle \mathfrak R^{T^0_q M} \hat Q,\hat Q \rangle = \sum_{1\leq i_1,\dots,i_q\leq m} \langle \mathfrak R Q^{(i_1,\dots,i_q)} , Q^{(i_1,\dots,i_q)} \rangle \, .
$$


\begin{remark}
	For a twice covariant tensor field $E = E_{ij} \, \theta^i \otimes \theta^j$ we have
	$$
		2 \hat E_{ijsr} = E_{sj} g_{ir} + E_{is} g_{jr} - E_{rj} g_{is} - E_{ir} g_{js} \, .
	$$
	If $E$ is symmetric and $\{\theta^i\}$ is chosen so that the dual frame $\{e_i\}$ is an orthonormal basis of eigenvectors of $E$ with corresponding eigenvalues $\eps_i$, then (no summation over $i$ or $j$ is intended)
	$$
	2 E^{(ij)} = 2 \hat E_{ijsr} \, \theta^s \otimes \theta^r = \eps_j \theta^j \otimes \theta^i + \eps_i \theta^i \otimes \theta^j - \eps_j \theta^i \otimes \theta^j - \eps_i \theta^j \otimes \theta^i = (\eps_i - \eps_j) \, \theta^i \wedge \theta^j \, .
	$$
	This shows that $\langle \mathfrak R^{T^0_2 M} \hat E,\hat E\rangle$ can be reduced to a linear combination of evaluations of $\langle \mathfrak R \,\cdot\, , \,\cdot\, \rangle$ on decomposable $2$-forms, that is, a linear combination of sectional curvatures. Namely,
	$$
	\langle \mathfrak R^{T^0_2 M} \hat E, \hat E \rangle = \frac{1}{4} \sum_{i,j=1}^m \langle \mathfrak R (\hat E_{ijsr} \, \theta^s \otimes \theta^r) , \hat E^{ij}_{\;\;\,kt} \, \theta^k \otimes \theta^t \rangle = \frac{1}{4} \sum_{i,j=1}^m (\eps_i - \eps_j)^2 \langle \mathfrak R (\theta^i \wedge \theta^j), \theta^i \wedge \theta^j \rangle
	$$
	that is
	$$
	\langle \mathfrak R^{T^0_2 M} \hat E, \hat E \rangle = \sum_{i,j=1}^m (\eps_i - \eps_j)^2 \Sect(e_i \wedge e_j) = \sum_{i,j=1}^m (\eps_i - \eps_j)^2 R_{ijij} \, .
	$$
	So, a lower bound on sectional curvatures of $M$ alone is sufficient to obtain lower bounds on $\langle \Gamma\,\cdot\, , \,\cdot\, \rangle$ acting on symmetric twice covariant tensors.
\end{remark}

\begin{remark}
	For a one-form $\omega = \omega_i \theta^i$ we have $2 \hat \omega_{isr} = \omega_s g_{ir} - \omega_r g_{is}$, hence
	$$
		2 \hat \omega_{isr} \, \theta^s \otimes \theta^r = g_{ir} \omega \otimes \theta^r - g_{is} \theta^s \otimes \omega = g_{ir} \omega \wedge \theta^r \, .
	$$
	Choosing $\{\theta^i\}$ as an orthonormal coframe with $\omega = |\omega|\theta^1$ we get
	$$
		\hat \omega = \frac{1}{2} \sum_{i=1}^m \theta^i \otimes (\omega \wedge \theta^i)
	$$
	and then
	$$
		\langle \mathfrak R^{T^\ast M} \hat \omega, \hat \omega \rangle = \frac{1}{4} \sum_{i=1}^m \langle \mathfrak R (\omega \wedge \theta^i), \omega \wedge \theta^i \rangle = |\omega|^2 \sum_{i=1}^m \Sect(e_1 \wedge e_i) = |\omega|^2 \Ricc(e_1,e_1) = \Ricc(\omega^\sharp,\omega^\sharp)
	$$
	with $\omega^\sharp$ the vector field metrically equivalent to $\omega$. Thus, a lower bound on $\Ricc$ is enough to have a lower bound on $\langle \Gamma\,\cdot\, , \,\cdot\, \rangle$ acting on one-forms.
\end{remark}

\subsection{Symmetric tensors and algebraic curvature tensors} \label{subsec_T}

Let $(M,\metric)$ be a Riemannian manifold of dimension $m>2$. We say that a $4$-covariant tensor field $T$ is an algebraic curvature tensor if it shares the symmetries of the Riemann curvature tensor and satisfies the first Bianchi identity. Namely, if $\{\theta^i\}_{i=1}^m$ is a local coframe on $M$ and
$$
	T = T_{ijkt} \, \theta^i \otimes \theta^j \otimes \theta^k \otimes \theta^t
$$
we require that
\begin{alignat}{2}
	\label{T_sym} & T_{ijkt} = - T_{jikt} = T_{ktij} && \qquad \forall \, 1 \leq i,j,k,t \leq m \, , \\
	\label{T_B1} & T_{ijkt} + T_{iktj} + T_{itkj} = 0 && \qquad \forall \, 1 \leq i,j,k,t \leq m \, .
\end{alignat}
We remark that \eqref{T_B1} is a consequence of \eqref{T_sym} if $m\leq 3$, see \cite[page 46]{bes08}.

If $T$ is a smooth algebraic tensor field, we say that $T$ satisfies the second Bianchi identity if
\begin{equation} \label{T_B2}
	T_{ijkt,l} + T_{ijlk,t} + T_{ijtl,k} = 0 \qquad \forall \, 1 \leq i,j,k,t,l \leq m \, .
\end{equation}
More generally, we can define a first-order differential operator $B : T \mapsto B(T)$ on the bundle of algebraic curvature tensors of $M$ by setting
$$
	B(T)(X,Y,Z,W,V) = (\nabla_V T)(X,Y,Z,W) + (\nabla_W T)(X,Y,V,Z) + (\nabla_Z T)(X,Y,W,V)
$$
for every $X,Y,Z,W,V\in\mathfrak X(M)$. In local notation this reads as
$$
	B(T)_{ijktl} = T_{ijkt,l} + T_{ijlk,t} + T_{ijtl,k}
$$
and $T$ satisfies the second Bianchi identity if and only if $B(T) = 0$.

\begin{definition}
	A smooth algebraic curvature tensor $T$ is harmonic if $\div T = 0$ and $B(T) = 0$.
\end{definition}

We let $E_T$ denote the Ricci contraction of $T$ defined by
$$
	E_T(X,Y) = \trace_g[(Z,W) \mapsto T(Z,X,W,Y)]
$$
for every $X,Y\in\mathfrak X(M)$. In local notation, $E_T = E_{ij} \, \theta^i \otimes \theta^j$ with
$$
	E_{ij} = T^k_{\;\;ikj} \, .
$$
We also set $S_T = \trace_g E_T$ and we denote $Z_T = E_T - \frac{S_T}{m}\metric$ the traceless part of $E_T$. We say that an algebraic curvature tensor is totally traceless if all of its contractions with the metric tensor vanish (equivalently, if its Ricci contraction is the zero tensor). Any algebraic curvature tensor $T$ can be orthogonally decomposed in a unique way as the sum
\begin{equation} \label{T_dec}
	T = W_T + V_T + U_T
\end{equation}
of a totally traceless Weyl part $W_T$ and two additional terms $V_T$ and $U_T$ that are further irreducible with respect to the action of the orthogonal group $O(m)$. Explicitely (see \cite{amr}),
\begin{equation} \label{VU_def}
	V_T = \frac{1}{m-2} Z_T \owedge \metric \, , \qquad U_T = \frac{S_T}{2m(m-1)} \metric \owedge \metric \, , \qquad W_T = T - V_T - U_T \, ,
\end{equation}
with $\owedge$ the Kulkarni-Nomizu product of symmetric bilinear forms. Setting
\begin{equation} \label{AT_def}
	A_T = E_T - \frac{S_T}{2(m-1)}\metric \equiv Z_T + \frac{m-2}{2m(m-1)} S_T \metric
\end{equation}
we can also write
\begin{equation} \label{TWA_dec}
	T = W_T + \frac{1}{m-2} A_T \owedge \metric \, .
\end{equation}
Note that $W_T$, $V_T$, $U_T$ and $A_T\owedge\metric$ also are algebraic curvature tensors. Moreover, if $m\leq 3$ then the Weyl part $W_T$ of $T$ is always zero, so that $T$ is completely determined by its Ricci contraction $E_T$, see \cite[observation 1.119.b)]{bes08}.

\begin{remark}
	For ease of notation, in the rest of this section and in the next Section \ref{sec_Boch} we drop the subscript ${}_T$ and we simply write $E$, $S$, $Z$, $A$, $W$, $V$, $U$ instead of $E_T$, $S_T$, $Z_T$, $A_T$, $W_T$, $V_T$, $U_T$ to denote the tensors associated to $T$ as above. This won't cause ambiguity with the notation that we adopted for the Weyl curvature tensor ($W$) and scalar curvature ($S$) of the manifold $(M,\metric)$, since these geometric objects will not appear in our analysis. On the other hand, we reserve the notation $R_{ijkt}$ and $R_{ij}$ for the components of the Riemann and Ricci curvature tensors of $(M,\metric)$. In Section \ref{sec_Tachi} we will resume to the use of the subscript ${}_T$, since also the Weyl curvature tensor and the scalar curvature of $M$ will come back into the play.
\end{remark}

%
\begin{lemma}
	For any algebraic curvature tensor $T$ we have
	\begin{align}
		\label{|T|_1}
		|T|^2 & = |W|^2 + \frac{4}{m-2} |Z|^2 + \frac{2S^2}{m(m-1)} \\
		\label{|DT|_1}
		|\nabla T|^2 & = |\nabla W|^2 + \frac{4}{m-2} |\nabla Z|^2 + \frac{2|\nabla S|^2}{m(m-1)} \\
		\intertext{or, equivalently,}
		\label{|T|_2}
		|T|^2 & = |W|^2 + \frac{4}{m-2}|E|^2 - \frac{2S^2}{(m-1)(m-2)} \\
		\label{|DT|_2}
		|\nabla T|^2 & = |\nabla W|^2 + \frac{4}{m-2}|\nabla E|^2 - \frac{2|\nabla S|^2}{(m-1)(m-2)} \, .
	\end{align}
\end{lemma}

\begin{proof}
	By orthogonality of the decomposition $T = W + V + U$ we have $|T|^2 = |W|^2 + |V|^2 + |U|^2$, then a direct computation yields
	$$
		|V|^2 = \frac{4}{m-2} |Z|^2 \, , \qquad |U|^2 = \frac{2S^2}{m(m-1)} \, .
	$$
	The second identity is proved by similar computations observing that $\nabla\metric = 0$. The third and fourth identities are equivalent to the first two since
	$$
		|E|^2 = |Z|^2 + \frac{S^2}{m} \, , \qquad |\nabla E|^2 = |\nabla Z|^2 + \frac{|\nabla S|^2}{m} \, .
	$$
\end{proof}

To any algebraic curvature tensor $T$ we can associate a $4$-covariant tensor $P = P_T$ of local components
\begin{equation} \label{P_def}
	P_{ijkt} = T_{ijkt} - \frac{1}{m-1} (g_{ik} E_{jt} - g_{it} E_{jk}) \, .
\end{equation}
Note that $P$ is not an algebraic curvature tensor. However, its definition is not accidental. In case $T = \Riem$ (thus, $E = \Ricc$) the $(1,3)$ version $P^i_{\;jkt} \, e_i \otimes \theta^j \otimes \theta^k \otimes \theta^t$ of $P$, of local components
$$
	P^i_{\;jkt} = R^i_{\;jkt} - \frac{1}{m-1} ( \delta^i_k R_{jt} - \delta^i_t R_{jk} ) \, ,
$$
is the projective curvature tensor, which is invariant under projective transformations and vanishes if and only if the manifold has constant sectional curvature. In general, we have

\begin{lemma}
	Let $T$ be an algebraic curvature tensor and let $P$ be as in \eqref{P_def}. Then
	\begin{equation} \label{|P|}
		|P|^2 = |T|^2 - \frac{2}{m-1} |E|^2 = |W|^2 + \frac{2m}{(m-2)(m-1)} |Z|^2 \, .
	\end{equation}
	In particular, $P = 0$ if and only if $T = \dfrac{S}{2m(m-1)} \metric \owedge \metric$.
\end{lemma}

\begin{proof}
	By direct computation,
	\begin{align*}
		|P|^2 & = |T|^2 - \frac{2 T_{ijkt}( g^{ik} E^{jt} - g^{it} E^{jk})}{m-1} + \frac{( g_{ik} E_{jt} - g_{it} E_{jk} )( g^{ik} E^{jt} - g^{it} E^{jk})}{(m-1)^2} \\
		& = |T|^2 - \frac{2}{m-1} |E|^2
	\end{align*}
	and substituting \eqref{|T|_1} and $|E|^2 = |Z|^2 + \frac{1}{m} S^2$ we obtain
	\begin{align*}
		|P|^2 & = |W|^2 + \frac{4}{m-2} |Z|^2 + \frac{2 S^2}{m(m-1)} - \frac{2}{m-1} \left( |Z|^2 - \frac{S^2}{m} \right) \\
		& = |W|^2 + \frac{2m}{(m-2)(m-1)} |Z|^2 \, .
	\end{align*}
\end{proof}

\subsection{Algebraic curvature tensors with $B(T) = 0$}


The condition $B(T) = 0$ has many relevant implications, that we briefly describe with the aim of establishing Propositions \ref{B(T)_cons1} and \ref{B(T)_cons} below. The arguments are essentially those that one applies when dealing with the case $T = \Riem$, where the condition $B(T) = 0$ is always satisfied, to deduce well known relations between the actions of several first order differential operators on the Riemann, Ricci, Weyl, Schouten and Einstein tensors of a Riemannian manifold.

First, let us recall that a symmetric twice covariant tensor field $E$ is a Codazzi tensor if
$$
	(\nabla_X E)(\,\cdot\,,Y) = (\nabla_Y E)(\,\cdot\,,X) \qquad \forall \, X,Y \in \mathfrak X(M) \, ,
$$
that is, if
$$
	E_{ij,k} - E_{ik,j} = 0 \qquad \forall \, 1 \leq i,j,k \leq m \, .
$$
More generally we can define a differential operator $C : E \mapsto C(E)$ on the bundle of symmetric twice covariant tensors by setting
$$
	C(E)(X,Y,Z) = (\nabla_Z E)(X,Y) - (\nabla_Y E)(X,Z) \qquad \forall \, X,Y,Z\in\mathfrak X(M) \, .
$$
In local notation, this reads as
$$
	C(E)_{ijk} = E_{ij,k} - E_{ik,j}
$$
and then $E$ is Codazzi if and only if $C(E) = 0$.

Let us assume that $T$ satisfies $B(T) = 0$. Tracing \eqref{T_B2} with respect to $i$ and $l$ we get
$$
	(\div T)_{jkt} = T^i_{\;jkt,i} = E_{jt,k} - E_{jk,t}
$$
hence $\div T = 0$ if and only if $E$ is a Codazzi tensor. Tracing again with respect to $j$ and $t$ we obtain the Schur's identity
$$
	2 E^i_{\;k,i} = S_k \, , \qquad \text{that is,} \qquad 2 \div E = \nabla S \, .
$$
Schur's identity is equivalent to the Einstein-like tensor $G = E - \frac{1}{2} S\metric$ being divergence-free. Equivalently, the Cotton-like tensor $C(A)$ of local components
\begin{equation} \label{C_def}
	C_{ijk} = A_{ij,k} - A_{ik,j}
\end{equation}
is totally trace-free,
\begin{equation} \label{C_tr0}
	C^i_{\;ji} = C^i_{\;ij} = C_{j\;i}^{\;i} = 0 \, .
\end{equation}
Writing
$$
	E_{ij,k} - E_{ik,j} = C_{ijk} + \frac{1}{2(m-1)}(S_k g_{ij} - S_j g_{ik}) \, ,
$$
\eqref{C_tr0} implies that the right-hand side is the sum of two orthogonal covariant tensors, hence it is apparent that $E$ is Codazzi if and only if $C = 0$ and $\nabla S = 0$. In particular,
\begin{equation} \label{divT_CS}
	|\div T|^2 = |C(A)|^2 + \frac{|\nabla S|^2}{2(m-1)} \, .
\end{equation}
Summarizing, \eqref{divT_CS} proves the validity of

\begin{proposition} \label{B(T)_cons1}
	Let $M$ be a Riemannian manifold of dimension $m\geq 3$ and let $T$ be a smooth algebraic curvature tensor satisfying the second Bianchi identity. Then
	$$
		\div T = 0 \qquad \Leftrightarrow \qquad C(A) = 0 \; \text{ and } \; \nabla S = 0 \, .
	$$
\end{proposition}

If $\dim M = 3$ then the Weyl part of any algebraic curvature tensor vanishes. If $m\geq 4$ then, as a second relevant consequence of $B(T) = 0$, there is a tight relation between $C = C(A)$, $B(W)$ and $\div W$, which allows to restate Proposition \eqref{B(T)_cons1} in a different form, see Proposition \ref{B(T)_cons} below. Writing \eqref{TWA_dec} in local notation we have
$$
	W_{ijkt} = T_{ijkt} - \frac{1}{m-2} (A_{ik} g_{jt} + A_{jt} g_{ik} - A_{it} g_{jk} - A_{jk} g_{it})
$$
then applying the operator $B$ to both sides and using $B(T) = 0$ and $\nabla\metric = 0$ we get
\begin{equation} \label{B(W)}
	B(W)_{ijktl} = - \frac{1}{m-2} \left( C_{ikl} g_{jt} + C_{ilt} g_{jk} + C_{itk} g_{jl} - C_{jkl} g_{it} - C_{jlt} g_{ik} - C_{jtk} g_{il} \right) \, .
\end{equation}
We trace with respect to $i$ and $l$. Since $W^i_{\;jik,t} = W^i_{\;jti,k} = 0$ as $W$ is totally traceless, we get
\begin{equation} \label{divW}
	W^i_{\;jkt,i} = B(W)_{ijkti} = \frac{m-3}{m-2} C_{jtk} \, , \qquad \text{that is,} \qquad \div W = - \frac{m-3}{m-2} C \, .
\end{equation}
Formulas \eqref{divW} and \eqref{B(W)} show that $C(A) = 0$ amounts to $\div W = 0$ and implies $B(W) = 0$. The converse is also true. To see this, we compute $|B(W)|^2$. Note that we can write
$$
	|B(W)|^2 = \frac{2}{(m-2)^2} (X_{ijktl} X^{ijktl} - X_{ijktl} X^{jiktl})
$$
with $X_{ijktl} = C_{ikl} g_{jt} + C_{ilt} g_{jk} + C_{itk} g_{jl}$. Then we have
\begin{align*}
	X_{ijktl} X^{ijktl} & = 3 C_{ikl} C^{ikl}  g_{jt} g^{jt} + 2 C_{ikl} C^{ilt} g_{jt} g^{jk} + 2 C_{ikl} C^{itk} g_{jt} g^{jl} + 2C_{ilt} C^{itk} g_{jk} g^{jl} \\
	& = 3m C_{ijl} C^{ijl} + 2 C_{ikl} C^{ilk} + 2C_{ikj} C^{ijk} + 2C_{ijt} C^{itj} \\
	& = 3(m-2) C_{ijk} C^{ijk} \, ,
\end{align*}
where we have used the symmetry $C_{ijk} = - C_{ikj}$, and
\begin{align*}
	X_{ijktl} X^{jiktl} & = (C_{ikl} g_{jt} + C_{ilt} g_{jk} + C_{itk} g_{jl})  g^{it} C^{jkl} \\
	& \phantom{=\;} + (C_{ikl} g_{jt} + C_{ilt} g_{jk} + C_{itk} g_{jl})  g^{ik} C^{jlt} \\
	& \phantom{=\;} + (C_{ikl} g_{jt} + C_{ilt} g_{jk} + C_{itk} g_{jl})  g^{il} C^{jtk} \\
	& = 3 C_{ikl} C^{ikl}
\end{align*}
where we have also exploited \eqref{C_tr0}. Summing up, we get
$$
	|B(W)|^2 = \frac{6(m-3)}{(m-2)^2} |C(A)|^2 \, , \qquad |\div W|^2 = \frac{(m-3)^2}{(m-2)^2} |C(A)|^2 \, .
$$
In conclusion, we have the following

\begin{proposition} \label{B(T)_cons}
	Let $M$ be a Riemannian manifold of dimension $m\geq 4$ and let $T$ be a smooth algebraic curvature tensor satisfying the second Bianchi identity. Then
	$$
		\div T = 0 \quad \Leftrightarrow \quad \div W = 0 \; \text{ and } \; \nabla S = 0
	$$
	and
	$$
		\div W = 0 \quad \Leftrightarrow \quad B(W) = 0 \quad \Leftrightarrow \quad C(A) = 0 \, .
	$$
	In particular, $T$ is harmonic if and only if $W$ is harmonic and $S$ is constant.
\end{proposition}

\section{Bochner identities and curvature bounds} \label{sec_Boch}

\subsection{Bochner identities}

Let $M$ be a Riemannian metric of dimension $m\geq 2$. For any smooth algebraic curvature tensor $T$ we have
\begin{equation} \label{Boc_T}
	\frac{1}{2}\Delta|T|^2 = |\nabla T|^2 + \frac{1}{2} \langle \Gamma T,T \rangle - \frac{1}{3}|B(T)|^2 - 2|\div T|^2 + \div X(T)
\end{equation}
where $X(T)$ is the vector field whose components along a local frame $\{e_i\}$ are given by
\begin{equation} \label{X_def}
	X(T)^i = T^{sjkt} B(T)^{\phantom{sjkt}\,i}_{sjkt} + 2 T^{ijkt} (\div T)_{jkt} \, .
\end{equation}
In particular, if $T$ is harmonic (i.e., $T$ satisfies the second Bianchi identity and $\div T = 0$) then
\begin{equation} \label{Boc_T_harm}
	\frac{1}{2}\Delta|T|^2 = |\nabla T|^2 + \frac{1}{2} \langle \Gamma T,T \rangle \, .
\end{equation}

%
%

A lower bound on the curvature operator $\mathfrak R$ yields a lower bound on $\langle \Gamma T,T \rangle$. In \cite{pw20}, Petersen and Wink showed that a lower bound on the partial trace $\mathfrak R^{(\lfloor\frac{m-1}{2}\rfloor)}$ suffices to obtain a lower bound on $\langle \Gamma T,T\rangle$ when the Ricci contraction $E$ of $T$ has the form $E = \frac{S}{m}\metric$, that is, when its traceless part $Z = 0$. In particular, by their result a lower bound on $\mathfrak R^{(\lfloor\frac{m-1}{2}\rfloor)}$ is enough to deduce a lower bound on $\langle \Gamma W,W \rangle$, where $W$ is the Weyl part of $T$.

In this section we will prove that a lower bound on $\mathfrak R^{(\lfloor\frac{m-1}{2}\rfloor)}$ in fact yields a lower bound on $\langle \Gamma T,T \rangle$ for any algebraic curvature tensor $T$, whitout further structural assumptions. This is done showing that
\begin{equation} \label{Gam_TT}
	\langle \Gamma T,T\rangle = \langle \Gamma W,W \rangle + \frac{4}{m-2} \langle \Gamma Z,Z\rangle
\end{equation}
(where in the last term the action of $\Gamma$ and $\metric$ is intended on the bundle $T^0_2 M$), and then estimating $\langle \Gamma Z,Z \rangle$ from below. Building on the ideas in \cite{pw20} and \cite{be69}, we show that a lower bound on $\langle \Gamma Z,Z \rangle$ can be established just assuming a lower bound on the sum of the sectional curvatures of any collection of $\lfloor\frac{m}{2}\rfloor$ mutually orthogonal $2$-planes in $TM$, and the latter is in turn implied by a lower bound on $\mathfrak R^{(k)}$ for some $k\leq \lfloor\frac{m}{2}\rfloor$, so in particular by a lower bound on $\mathfrak R^{(\lfloor\frac{m-1}{2}\rfloor)}$. Putting together the lower bounds on $\langle \Gamma W,W \rangle$ and $\langle \Gamma Z,Z\rangle$ we shall see that
$$
	\frac{1}{2} \langle \Gamma T,T \rangle \geq (m-1)C|P|^2
$$
provided $\mathfrak R^{(\lfloor\frac{m-1}{2}\rfloor)} \geq C$, where $P$ is the pseudo-projective curvature tensor defined in \eqref{P_def}. Note that for $T = \Riem$ this is precisely the estimate given by Tachibana in \cite{tachi74} under the stronger assumption $\mathfrak R \geq C$.

Hence, the main goal of this section will be the proof of the following

\begin{theorem} \label{Boch_T}
	Let $M$ be a Riemannian manifold of dimension $m\geq 2$ with $\mathfrak R^{(\lfloor\frac{m-1}{2}\rfloor)} \geq a(x)$ for some function $a : M \to \R$ and let $T$ be a smooth algebraic curvature tensor. Then
	\begin{equation}
		\frac{1}{2} \Delta|T|^2 \geq |\nabla T|^2 + (m-1)a(x)|P|^2 - \frac{1}{3}|B(T)|^2 - 2|\div T|^2 + \div X(T)
	\end{equation}
	where $P$ and $X(T)$ are as in \eqref{P_def} and \eqref{X_def}. If $T$ is harmonic,
	\begin{equation}
		\frac{1}{2} \Delta|T|^2 \geq |\nabla T|^2 + (m-1)a(x)|P|^2 \, .
	\end{equation}
\end{theorem}

\subsection{Proofs of \eqref{Boc_T} and \eqref{Gam_TT}}

We start with the proof of the Bochner-type identity \eqref{Boc_T}.

\begin{proposition} \label{Boch_T_prop1}
	Let $(M,\metric)$ be a Riemannian manifold and let $T$ be a smooth algebraic curvature tensor. Then
	$$
		\frac{1}{2} \Delta|T|^2 = |\nabla T|^2 + \frac{1}{2} \langle \Gamma T,T \rangle - \frac{1}{3}|B(T)|^2 - 2|\div T|^2 - \div X(T)
	$$
	where $X(T)$ is the vector field given by \eqref{X_def}.
\end{proposition}

\begin{proof}
	We compute
	$$
		\frac{1}{2}\Delta|T|^2 = \div(\nabla|T|^2) = (T^{ijkt}T^{\phantom{ijkt,}\!l}_{ijkt,})_{,l} = T^{ijkt}_{\phantom{ijkt},l} T^{\phantom{ijkt,}\!l}_{ijkt,} + T^{ijkt}T^{\phantom{ijkt,}\!l}_{ijkt,\;l}
	$$
	and $T^{ijkt}_{\phantom{ijkt},l} T^{\phantom{ijkt,}\!l}_{ijkt,} = |\nabla T|^2$. Looking at the second term, we rewrite
	$$
		T^{\phantom{ijkt,}\!l}_{ijkt,\;l} = T^{\phantom{ijkt,}\!l}_{ijkt,\;l} + T^{\phantom{ij}l}_{ij\;k,tl} + T^{\phantom{ijt}l}_{ijt\;,kl} - T^{\phantom{ij}l}_{ij\;k,tl} - T^{\phantom{ijt}l}_{ijt\;,kl} = B(T)^{\phantom{ijkt}\,l}_{ijkt\;,l} + T^{\phantom{ijk}l}_{ijk\;,tl} - T^{\phantom{ijt}l}_{ijt\;,kl}
	$$
	so that, using the symmetry $T^{ijtk} = - T^{ijkt}$,
	$$
		T^{ijkt}T^{\phantom{ijkt,}\!l}_{ijkt,\;l} = T^{ijkt} B(T)^{\phantom{ijkt}\,l}_{ijkt\;,l} + 2 T^{ijkt} T^{\phantom{ijk}l}_{ijk\;,tl} \, .
	$$
	We further rewrite
	$$
		T^{\phantom{ijk}l}_{ijk\;,tl} = T^{\phantom{ijk}l}_{ijk\;,lt} + T^{\phantom{ijk}l}_{ijk\;,tl} - T^{\phantom{ijk}l}_{ijk\;,lt}
	$$
	and summing up we obtain
	\begin{equation} \label{BT0}
		\frac{1}{2} \Delta|T|^2 = |\nabla T|^2 + T^{ijkt} B(T)^{\phantom{ijkt}\,l}_{ijkt\;,l} + 2 T^{ijkt} T^{\phantom{ijk}l}_{ijk\;,lt} + 2 T^{ijkt} ( T^{\phantom{ijk}l}_{ijk\;,tl} - T^{\phantom{ijk}l}_{ijk\;,lt} ) \, .
	\end{equation}
	``Integrating by parts'' we get
	\begin{align*}
		T^{ijkt} B(T)^{\phantom{ijkt}\,l}_{ijkt\;,l} & = \div \left( T^{ijkt} B(T)^{\phantom{ijkt}\,l}_{ijkt} \, e_l \right) - T^{ijkt,l} B(T)_{ijktl} \\
		& = \div \left( T^{ijkt} B(T)^{\phantom{ijkt}\,l}_{ijkt} \, e_l \right) - \frac{1}{3} |B(T)|^2 \, , \\
		T^{ijkt} T^{\phantom{ijk}l}_{ijk\;,lt} & = \div \left( T^{ijkt} T^{\phantom{ijk}l}_{ijk\;,l} \, e_t \right) - T^{ijkt}_{\phantom{ijkt},t} T^{\phantom{ijk}l}_{ijk\;,l} \\
		& = \div \left( T^{ijkt} T^{\phantom{ijk}l}_{ijk\;,l} \, e_t \right) - |\div T|^2
	\end{align*}
	hence
	\begin{equation} \label{BT1}
		T^{ijkt} B(T)^{\phantom{ijkt,}l}_{ijkt,\;l} + 2 T^{ijkt} T^{\phantom{ijk}l}_{ijk\;,lt} = \div X(T) - \frac{1}{3} |B(T)|^2 - 2 |\div T|^2 \, .
	\end{equation}
	On the other hand, by the symmetries of $T$ and \eqref{Gam_Ric} we have
	\begin{align*}
		4 T^{ijkt} ( T^{\phantom{ijk}l}_{ijk\;,tl} - T^{\phantom{ijk}l}_{ijk\;,lt} ) & = T^{ijkt} ( T^{\phantom{ijk}l}_{ijk\;,tl} - T^{\phantom{ijk}l}_{ijk\;,lt} ) + T^{ijtk} ( T^{\phantom{ij}l}_{ij\;k,tl} - T^{\phantom{ij}l}_{ij\;k,lt} ) \\
		& \phantom{=\;} + T^{ktij} ( T^{\;l}_{k\;ij,tl} - T^{\;l}_{k\;ij,lt} ) + T^{tkij} ( T^l_{\;kij,tl} - T^l_{\;kij,lt} ) \\
		& = \frac{1}{2} \langle \Gamma T,T \rangle \, .
	\end{align*}
	Substituting this and \eqref{BT1} into \eqref{BT0} we obtain the desired conclusion.
\end{proof}

We now turn to \eqref{Gam_TT}, that is a consequence of the following

\begin{proposition} \label{Gam_TT_gen}
	Let $T$, $\tilde T$ be algebraic tensor fields. Then
	$$
		\langle \Gamma T,\tilde T \rangle = \langle \Gamma W,\tilde W \rangle + \frac{4}{m-2} \langle \Gamma Z,\tilde Z \rangle
	$$
	where $W$, $\tilde W$ are the Weyl parts of $T$, $\tilde T$ and $Z$, $\tilde Z$ are the traceless parts of their respective Ricci contractions $E$, $\tilde E$.
\end{proposition}

The proof of Proposition \ref{Gam_TT_gen} is essentially a long computation, that we split into the proofs of several lemmas.

\begin{lemma}
	Let $T,\tilde T$ be algebraic curvature tensors. Then
	\begin{equation} \label{Gam_TT0}
		\langle \Gamma T,\tilde T \rangle = 4 R_{is} T^s_{\;\;jkt} \tilde T^{ijkt} - 4 R_{isjl} T^{sl}_{\;\;\;kt} \tilde T^{ijkt} - 8 R_{iskl} T^{s\;l}_{\;\;j\;t} \tilde T^{ijkt} \, .
	\end{equation}
\end{lemma}

\begin{proof}
	From the very definition \eqref{Gam_def} we have
	\begin{align*}
		(\Gamma T)_{ijkt} & = R_{is} T^s_{\;\,jkt} + R_{js} T^{\,s}_{i\;kt} + R_{ks} T^{\;\;s}_{ij\;t} + R_{ts} T^{\;\;\;\;s}_{ijk} \\
		& \phantom{=\;} - R_{isjl} T^{sl}_{\;\;\;kt} - R_{jsil} T^{ls}_{\;\;\;kt} - R_{iskl} T^{s\;l}_{\;\,j\;t} - R_{ksil} T^{l\;s}_{\;j\;t} - R_{istl} T^{s\;\;\,l}_{\;\,jk} - R_{tsil} T^{l\;\;\,s}_{\;jk} \\
		& \phantom{=\;} - R_{jskl} T^{\,sl}_{i\;\;t} - R_{ksjl} T^{\,ls}_{i\;\;t} - R_{jstl} T^{\,s\;l}_{i\;k} - R_{tsjl} T^{\,l\;s}_{i\;k} - R_{kstl} T^{\;\;\,sl}_{ij} - R_{tskl} T^{\;\;\,ls}_{ij} \, .
	\end{align*}
	We contract with $\tilde T^{ijkt}$. Using the symmetries of $T$, $\tilde T$ and renaming indexes we get
	\begin{align*}
		& R_{is} T^s_{\;\,jkt} \tilde T^{ijkt} + R_{js} T^{\,s}_{i\;kt} \tilde T^{ijkt} + R_{ks} T^{\;\;s}_{ij\;t} \tilde T^{ijkt} + R_{ts} T^{\;\;\;\;s}_{ijk} \tilde T^{ijkt} \\
		= \; & R_{is} T^s_{\;\,jkt} \tilde T^{ijkt} + R_{js} T^{s}_{\;ikt} \tilde T^{jikt} + R_{ks} T^{s}_{\;tij} \tilde T^{ktij} + R_{ts} T^s_{\;\,kij} \tilde T^{tkij} \\
		= \; & 4 R_{is} T^s_{\;\,jkt} \tilde T^{ijkt} \, , \\
		& R_{isjl} T^{sl}_{\;\;\;kt} \tilde T^{ijkt} + R_{jsil} T^{ls}_{\;\;\;kt} \tilde T^{ijkt} + R_{kstl} T^{\;\;\,sl}_{ij} \tilde T^{ijkt} + R_{tskl} T^{\;\;\,ls}_{ij} \tilde T^{ijkt} \\
		= \; & R_{isjl} T^{sl}_{\;\;\;kt} \tilde T^{ijkt} + R_{jsil} T^{sl}_{\;\;\;kt} \tilde T^{jikt} + R_{kstl} T^{sl}_{\;\;\,ij} \tilde T^{ktij} + R_{tskl} T^{sl}_{\;\;\,ij} \tilde T^{tkij} \\
		= \; & 4 R_{isjl} T^{sl}_{\;\;\;kt} \tilde T^{ijkt} \, , \\
		& R_{iskl} T^{s\;l}_{\;\,j\;t} \tilde T^{ijkt} + R_{ksil} T^{l\;s}_{\;j\;t} \tilde T^{ijkt} + R_{istl} T^{s\;\;\,l}_{\;\,jk} \tilde T^{ijkt} + R_{tsil} T^{l\;\;\,s}_{\;jk} \tilde T^{ijkt} \\
		= \; & R_{iskl} T^{s\;l}_{\;\,j\;t} \tilde T^{ijkt} + R_{ksil} T^{s\;l}_{\;\,t\;j} \tilde T^{ktij} + R_{istl} T^{s\;l}_{\;\,j\;k} \tilde T^{ijtk} + R_{tsil} T^{s\;l}_{\;\,k\;j} \tilde T^{tkij} \\
		= \; & 4 R_{iskl} T^{s\;l}_{\;\,j\;t} \tilde T^{ijkt} \, , \\
		& R_{jskl} T^{\,sl}_{i\;\;t} \tilde T^{ijkt} + R_{ksjl} T^{\,ls}_{i\;\;t} \tilde T^{ijkt} + R_{jstl} T^{\,s\;l}_{i\;k} \tilde T^{ijkt} + R_{tsjl} T^{\,l\;s}_{i\;k} \tilde T^{ijkt} \\
		= \; & R_{jskl} T^{s\;l}_{\;\,i\;t} \tilde T^{jikt} + R_{ksjl} T^{s\;l}_{\;\,t\;i} \tilde T^{ktji} + R_{jstl} T^{s\;l}_{\;\,i\;k} \tilde T^{jitk} + R_{tsjl} T^{s\;l}_{\;\,k\;i} \tilde T^{tkji} \\
		= \; & 4 R_{jskl} T^{s\;l}_{\;\,i\;t} \tilde T^{jikt} = 4 R_{iskl} T^{s\;l}_{\;\,j\;t} \tilde T^{ijkt} \, .
	\end{align*}
	Summing up we obtain \eqref{Gam_TT0}.
\end{proof}

In the next three Lemmas we denote by $g = \metric$ the metric tensor of $M$.

\begin{lemma}
	Let $W$ be a totally traceless algebraic curvature tensor and $E$ a symmetric $2$-covariant tensor. Then
	\begin{equation} \label{Gam_TE}
		\langle \Gamma W, E\owedge g \rangle = 0 \, .
	\end{equation}
\end{lemma}

\begin{proof}
	We apply \eqref{Gam_TT0} with $T = W$ and $\tilde T = E\owedge g$. We write
	$$
		(E\owedge g)_{ijkt} = E_{ik} g_{jt} + E_{jt} g_{ik} - E_{it} g_{jk} - E_{jk} g_{it} \, .
	$$
	We separately compute the three terms in \eqref{Gam_TT0}. Since $W$ is totally traceless, we have
	\begin{align*}
		R_{is} W^s_{\;\;jkt} (E\owedge g)^{ijkt} & = R_{is} W^s_{\;\;jkt} E^{jt} g^{ik} - R_{is} W^s_{\;\;jkt} E^{jk} g^{it} \\
		& = R_{is} W^{s\;i}_{\;\,j\;t} E^{jt} - R_{is} W^{s\;\;i}_{\;\,jk} E^{jk} \\
		& = 2 R_{is} W^{s\;i}_{\;\,j\;t} E^{jt} \, , \\
		R_{isjl} W^{sl}_{\;\;\;kt} (E\owedge g)^{ijkt} & = R_{isjl} W^{sl}_{\;\;\;kt} (E^{ik} g^{jt} - E^{it} g^{jk}) + R_{isjl} W^{sl}_{\;\;\;kt} (E^{jt} g^{ik} - E^{jk} g^{it}) \\
		& = 2 R_{isjl} W^{sl}_{\;\;\;kt} E^{ik} g^{jt} + 2 R_{isjl} W^{sl}_{\;\;\;kt} E^{jt} g^{ik} \\
		& = 2 R_{isjl} W^{sl\;j}_{\;\;\;k} E^{ik} + 2 R_{isjl} W^{sli}_{\;\;\;\;t} E^{jt} \, , \\
		R_{iskl} W^{s\;l}_{\;\;j\;t} (E\owedge g)^{ijkt} & = R_{iskl} W^{s\;l}_{\;\;j\;t} E^{jt} g^{ik} - R_{iskl} W^{s\;l}_{\;\;j\;t} E^{it} g^{jk} - R_{iskl} W^{s\;l}_{\;\;j\;t} E^{jk} g^{it} \\
		& = R_{sl} W^{s\;l}_{\;\;j\;t} E^{jt} - R_{iskl} W^{skl}_{\;\;\;\;\,t} E^{it} - R_{iskl} W^{s\;li}_{\;\;j} E^{jk} \, .
	\end{align*}
	Summing up, we obtain
	\begin{align*}
		\frac{1}{8} \langle \Gamma W, E\owedge g \rangle & = R_{is} W^{s\;i}_{\;\,j\;t} E^{jt} - R_{isjl} W^{sl\;j}_{\;\;\;k} E^{ik} - R_{isjl} W^{sli}_{\;\;\;\;t} E^{jt} \\
		& \phantom{=\;} - R_{sl} W^{s\;l}_{\;\;j\;t} E^{jt} + R_{iskl} W^{skl}_{\;\;\;\;\,t} E^{it} + R_{iskl} W^{s\;li}_{\;\;j} E^{jk} \, .
	\end{align*}
	A few algebraic manipulations yield
	\begin{align*}
		R_{is} W^{s\;i}_{\;\,j\;t} E^{jt} & \equiv R_{ls} W^{s\;l}_{\;\,j\;t} E^{jt} = R_{sl} W^{s\;l}_{\;\,j\;t} E^{jt} \, , \\
		R_{isjl} W^{sl\;j}_{\;\;\;k} E^{ik} & = R_{jlis} W^{lsj}_{\;\;\;\;k} E^{ik} \equiv R_{isjl} W^{sli}_{\;\;\;\;t} E^{jt} \\
		R_{iskl} W^{skl}_{\;\;\;\;\,t} E^{it} & = R_{klis} W^{l\;sk}_{\;\;t} E^{ti} \equiv R_{iskl} W^{s\;li}_{\;\;j} E^{jk} \, ,
	\end{align*}
	where ``$\equiv$'' denotes mere renaming of indexes while ``$=$'' indicates the use of some symmetry of $\Ricc$, $\Riem$, $W$ or $E$. Substituting and manipulating a little more we get
	\begin{align*}
		\frac{1}{16} \langle \Gamma W, E\owedge g \rangle & = R_{iskl} W^{skl}_{\;\;\;\;\,t} E^{it} - R_{isjl} W^{sl\;j}_{\;\;\;k} E^{ik} \\
		& = R_{iskl} W^{skl}_{\;\;\;\;\,t} E^{it} + R_{isjl} W^{slj}_{\;\;\;\;\,k} E^{ik} \\
		& \equiv R_{iskl} W^{skl}_{\;\;\;\;\,t} E^{it} + R_{islk} W^{skl}_{\;\;\;\;\,t} E^{it} \\
		& = (R_{iskl} + R_{islk}) W^{skl}_{\;\;\;\;\,t} E^{it} \\
		& = 0 \, .
	\end{align*}
\end{proof}

\begin{lemma}
	Let $E,\tilde E$ be symmetric $2$-covariant tensors and let $Z,\tilde Z$ be their respective traceless parts. Then
	\begin{equation} \label{Gam_EE0}
		\langle \Gamma Z,\tilde Z \rangle = \langle \Gamma E,\tilde E \rangle = 2 R_{is} E^s_{\;j} \tilde E^{ij} - 2 R_{isjl} E^{ls} \tilde E^{ij} \, .
	\end{equation}
\end{lemma}

\begin{proof}
	From \eqref{Gam_def} we have
	\begin{equation} \label{Gam_E}
		(\Gamma E)_{ij} = R_{is} E^s_{\;j} + R_{js} E_i^{\;s} - 2 R_{isjl} E^{sl} \, .
	\end{equation}
	Contracting with $\tilde E^{ij}$ we get the second equality \eqref{Gam_EE0}. In case $E = g$, from \eqref{Gam_E} we deduce $\Gamma g = 0$. The first equality in \eqref{Gam_EE0} then follows by linearity and self-adjointness of $\Gamma$ on $T^0_2 M$.
\end{proof}

\begin{lemma}
	Let $E,\tilde E$ be symmetric $2$-covariant tensors and let $Z,\tilde Z$ be their respective traceless parts. Then
	\begin{equation} \label{Gam_EEKN}
		\langle \Gamma(E\owedge g),\tilde E\owedge g \rangle = 4(m-2) \langle \Gamma E,\tilde E \rangle = 4(m-2) \langle \Gamma Z,\tilde Z \rangle \, .
	\end{equation}
\end{lemma}

\begin{proof}
	We separately compute the three terms in \eqref{Gam_TT0} for $T = E\owedge g$, $\tilde T = \tilde E\owedge g$. We have
	\begin{align*}
		(E\owedge g)^{sl}_{\;\;\;kt} (\tilde E\owedge g)^{ijkt} & = E^s_{\;k} (\tilde E\owedge g)^{ijkl} - E^s_{\;t} (\tilde E\owedge g)^{ijlt} + E^l_{\;t} (\tilde E\owedge g)^{ijst} - E^l_{\;k} (\tilde E\owedge g)^{ijks} \\
		& = 2 E^s_{\;k} (\tilde E\owedge g)^{ijkl} - 2 E^l_{\;k} (\tilde E\owedge g)^{ijks} \\
		& = 2 E^s_{\;k} \tilde E^{ik} g^{jl} + 2 E^{si} \tilde E^{jl} - 2 E^s_{\;k} \tilde E^{jk} g^{il} - 2 E^{sj} \tilde E^{il} \\
		& \phantom{=\;} - 2 E^l_{\;k} \tilde E^{ik} g^{js} - 2 E^{li} \tilde E^{js} + 2 E^l_{\;k} \tilde E^{jk} g^{is} + 2 E^{lj} \tilde E^{is} \, .
	\end{align*}
	Hence,
	\begin{align*}
		R_{is} (E\owedge g)^s_{\;\;jkt} & (\tilde E\owedge g)^{ijkt} = R_{is} g_{lj} (E\owedge g)^{sl}_{\;\;\;kt} (\tilde E\owedge g)^{ijkt} \\
		& = 2 m R_{is} E^s_{\;k} \tilde E^{ik} + 2 R_{is} E^{si} \tilde E^j_{\;j} - 2 R_{is} E^s_{\;k} \tilde E^{ik} - 2 R_{is} E^s_{\;l} \tilde E^{il} \\
		& \phantom{=\;} - 2 R_{il} E^l_{\;k} \tilde E^{ik} - 2 R_{is} E^{\;i}_j \tilde E^{js} + 2 R^i_{\;i} E_{jk} \tilde E^{jk} + 2 R_{is} \tilde E^{is} E^j_{\;j} \\
		& = 2 (m-4) R_{ij} E^j_{\;k} \tilde E^{ik} + 2 R_{ij} E^{ij} \tilde E^k_{\;k} + 2 R_{ij} \tilde E^{ij} E^k_{\;k} + 2 R^i_{\;i} E_{jk} \tilde E^{jk} \, , \\
		R_{isjl} (E\owedge g)^{sl}_{\;\;\;kt} & (\tilde E\owedge g)^{ijkt} = 4 R_{sj} E^s_{\;k} \tilde E^{jk} + 4 R_{islj} E^{sj} \tilde E^{il} \, . \\
	\end{align*}
	Similarly
	\begin{align*}
		(E\owedge g)^{s\;l}_{\;j\;t} (\tilde E\owedge g)^{ijkt} & = E^{sl} (\tilde E\owedge g)^{ijk}_{\;\;\;\;j} - E^s_{\;t} (\tilde E\owedge g)^{ilkt} - E_j^{\;l} (\tilde E\owedge g)^{ijks} + E_{jt} g^{sl} (\tilde E\owedge g)^{ijkt} \\
		& = E^{sl} \tilde E^{ik} g^j_{\;j} + E^{sl} \tilde E^j_{\;j} g^{ik} - E^{sl} \tilde E^i_{\;j} g^{jk} - E^{sl} \tilde E^{jk} g^i_{\;j} \\
		& \phantom{=\;} - E^s_{\;t} \tilde E^{ik} g^{lt} - E^s_{\;t} \tilde E^{lt} g^{ik} + E^s_{\;t} \tilde E^{it} g^{lk} + E^s_{\;t} \tilde E^{lk} g^{it} \\
		& \phantom{=\;} - E^{\;l}_j \tilde E^{ik} g^{js} - E^{\;l}_j \tilde E^{js} g^{ik} + E^{\;l}_j \tilde E^{is} g^{jk} + E^{\;l}_j \tilde E^{jk} g^{is} \\
		& \phantom{=\;} + E_{jt} \tilde E^{ik} g^{sl} g^{jt} + E_{jt} \tilde E^{jt} g^{sl} g^{ik} - E_{jt} \tilde E^{it} g^{sl} g^{jk} - E_{jt} \tilde E^{jk} g^{sl} g^{it} \\
		& = m E^{sl} \tilde E^{ik} + E^{sl} \tilde E^j_{\;j} g^{ik} - E^{sl} \tilde E^{ik} - E^{sl} \tilde E^{ik} \\
		& \phantom{=\;} - E^{sl} \tilde E^{ik} - E^s_{\;t} \tilde E^{lt} g^{ik} + E^s_{\;t} \tilde E^{it} g^{lk} + E^{si} \tilde E^{lk} \\
		& \phantom{=\;} - E^{sl} \tilde E^{ik} - E^l_{\;j} \tilde E^{js} g^{ik} + E^{kl} \tilde E^{is} + E^l_{\;j} \tilde E^{jk} g^{is} \\
		& \phantom{=\;} + \tilde E^{ik} E^j_{\;j} g^{sl} + E_{jt} \tilde E^{jt} g^{sl} g^{ik} - E^k_{\;t} \tilde E^{it} g^{sl} - E^i_{\;j} \tilde E^{jk} g^{sl}
	\end{align*}
	that is
	\begin{align*}
		(E\owedge g)^{s\;l}_{\;j\;t} & (\tilde E\owedge g)^{ijkt} = (m-4) E^{sl} \tilde E^{ik} - E^s_{\;t} \tilde E^{lt} g^{ik} - E^l_{\;j} \tilde E^{js} g^{ik} - E^k_{\;t} \tilde E^{it} g^{sl} - E^i_{\;j} \tilde E^{jk} g^{sl} \\
		& \phantom{=\;} + E^{sl} \tilde E^j_{\;j} g^{ik} + \tilde E^{ik} E^j_{\;j} g^{sl} + E_{jt} \tilde E^{jt} g^{sl} g^{ik} + E^s_{\;t} \tilde E^{it} g^{lk} + E^{si} \tilde E^{lk} + E^{kl} \tilde E^{is} + E^l_{\;j} \tilde E^{jk} g^{is} \, .
	\end{align*}
	We contract with $R_{iskl}$ to get
	\begin{align*}
		R_{iskl} (E\owedge g)^{s\;l}_{\;j\;t} (\tilde E\owedge g)^{ijkt} & = (m-4) R_{iskl} E^{sl} \tilde E^{ik} - 4 R_{sl} E^s_{\;t}\tilde E^{lt} \\
		& \phantom{=\;} + R_{sl} E^{sl} \tilde E^j_{\;j} + R_{ik} \tilde E^{ik} E^j_{\;j} + R^i_{\;i} E_{jt} \tilde E^{jt}
	\end{align*}
	Summing up,
	\begin{align*}
		\langle \Gamma(E\owedge g),\tilde E\owedge g \rangle & = 8 (m-4) R_{ij} E^j_{\;k} \tilde E^{ik} + 8 R_{ij} E^{ij} \tilde E^k_{\;k} + 8 R_{ij} \tilde E^{ij} E^k_{\;k} + 8 R^i_{\;i} E_{jk} \tilde E^{jk} \\
		& \phantom{=\;} - 16 R_{sj} E^s_{\;k} \tilde E^{jk} - 16 R_{islj} E^{sj} \tilde E^{il} - 8(m-4)R_{iskl} E^{sl} \tilde E^{ik} + 32 R_{sl} E^s_{\;t}\tilde E^{lt} \\
		& \phantom{=\;} - 8 R_{sl} E^{sl} \tilde E^j_{\;j} - 8 R_{ik} \tilde E^{ik} E^j_{\;j} - 8 R^i_{\;i} E_{jt} \tilde E^{jt} \\
		& = 8(m-2) R_{ij} E^j_{\;k} \tilde E^{ik} - 8(m-2) R_{iskl} E^{sl} \tilde E^{ik}
	\end{align*}
	and by \eqref{Gam_EE0} we obtain \eqref{Gam_EEKN}.
\end{proof}


\begin{proof}[Proof of Proposition \ref{Gam_TT_gen}]
	As in \eqref{TWA_dec}, we write
	$$
		T = W + \frac{1}{m-2} A\owedge g \, , \qquad \tilde T = \tilde W + \frac{1}{m-2} \tilde A \owedge g
	$$
	where $A$, $\tilde A$ are the Schouten-like tensors associated to $T$, $\tilde T$ as in \eqref{AT_def}. Then we apply the previous Lemmas with the choices $E = \frac{1}{m-2} A$, $\tilde E = \frac{1}{m-2} \tilde A$, noting that $Z$, $\tilde Z$ are also the traceless parts of $A$, $\tilde A$.
\end{proof}

\subsection{Lower bounds on $\langle \Gamma \,\cdot\, , \,\cdot\, \rangle$}

We now proceed to prove that a lower bound on $\mathfrak R^{(\lfloor\frac{m-1}{2}\rfloor)}$ ensures a lower bound on $\langle \Gamma T,T \rangle$ for any algebraic curvature tensor $T$. Among the results stated and proved in this subsection, our original contribution is represented by Proposition \ref{ZRic} and Theorem \ref{thm_m-1/2_R}. Lemma \ref{lem_That} is formally equivalent to Lemma 3 in \cite{tachi74}, while the statements of Lemma \ref{lem_PW_1} and Proposition \ref{m-1/2_R} are covered by Lemma 2.1, Lemma 2.2(c) and Proposition 2.5(b) of \cite{pw20}. Since the formalism (as well as the choice of normalization constants in the definition of the norm of an anti-symmetric tensor) adopted here differs from that of \cite{pw20}, we provide self-contained proofs along the lines of those in \cite{pw20} for ease of the reader.

\begin{lemma}
	Let $N\geq2$ be a positive integer and let $\{a_i\}_{1\leq i\leq N}$, $\{b_i\}_{1\leq i\leq N}$ be sequences of non-negative real numbers such that
	\begin{equation} \label{ab_ass}
		a_i \leq a_{i+1} \quad \text{for } \, 1 \leq i < N \qquad \text{and} \qquad b_i \geq 0 \quad \text{for } \, i = 1,\dots,N \, .
	\end{equation}
	Let $1 \leq k < N$ be an integer such that
	\begin{equation} \label{b_pinch}
		b_i \leq \frac{1}{k} \sum_{j=1}^N b_j \qquad \text{for } \, i = 1,\dots, N \, .
	\end{equation}
	Then
	\begin{equation} \label{ab_lower}
		\sum_{i=1}^N a_i b_i \geq \frac{1}{k} \sum_{i=1}^{k} a_i \sum_{j=1}^N b_j \, .
	\end{equation}
\end{lemma}

\begin{proof}
	We separately estimate
	\begin{align*}
		\sum_{i=k+1}^N a_i b_i & \geq a_{k+1} \sum_{i=k+1}^N b_i \, , \\
		\sum_{i=1}^k a_i b_i & = \sum_{i=1}^k (a_i - a_{k+1})b_i + a_{k+1} \sum_{i=1}^k b_i \geq \frac{1}{k} \sum_{i=1}^k (a_i - a_{k+1}) \sum_{j=1}^N b_j + a_{k+1} \sum_{i=1}^k b_i \, ,
	\end{align*}
	where we have used \eqref{b_pinch} and the fact that $a_i - a_{k+1} \leq 0$ for $i \leq k$. Summing up,
	$$
		\sum_{i=1}^N a_i b_i \geq \left[ \frac{1}{k} \sum_{i=1}^k (a_i - a_{k+1}) + a_{k+1} \right] \sum_{j=1}^N b_j = \left[ \frac{1}{k} \sum_{i=1}^k a_i \right] \sum_{j=1}^N b_j \, .
	$$
\end{proof}

\begin{remark} \label{rem_ab}
	If \eqref{b_pinch} holds with $1 \leq k < N$ a real number then it also holds with $k$ replaced by $\lfloor k\rfloor$, which is an integer in the range $\{1,\dots,N-1\}$. Hence, if $\{a_i\},\{b_i\}$ are as in \eqref{ab_ass} then for any real number $1 \leq k < N$ we have the implication
	$$
		b_i \leq \frac{1}{k} \sum_{j=1}^N b_j \quad \forall \, i = 1,\dots,N \qquad \Rightarrow \qquad \sum_{i=1}^N a_i b_i \geq \frac{1}{\lfloor k\rfloor} \sum_{i=1}^{\lfloor k\rfloor} a_i \sum_{j=1}^N b_j \, .
	$$
\end{remark}

\begin{proposition} \label{ZRic}
	Let $x\in M$, $C\in\R$ and assume that for every collection $\{\pi_1,\dots,\pi_{\lfloor\frac{m}{2}\rfloor}\}$ of mutually orthogonal $2$-dimensional subspaces of $T_x M$ it holds
	\begin{equation} \label{m2_sect}
		\frac{1}{\lfloor\frac{m}{2}\rfloor} \sum_{i=1}^{\lfloor\frac{m}{2}\rfloor} \Sect(\pi_i) \geq C \, .
	\end{equation}
	Then for any traceless symmetric $2$-covariant tensor $Z$ we have
	\begin{equation}
		\langle \Gamma Z,Z \rangle \geq 2mC|Z|^2 \qquad \text{at } \, x .
	\end{equation}
\end{proposition}

\begin{proof}
	Consider a coframe $\{\theta^i\}$ whose dual frame $\{e_i\}$ consists of eigenvectors of $Z$, with corresponding eigenvalues $\{\zeta_i\}_{1\leq i\leq m}$. From \eqref{Gam_E} we obtain (no summation is intended on $i$)
	$$
		(\Gamma Z)_{ii} = 2 R_{ii} \zeta_i - 2 \sum_{j=1}^m R_{ijij} \zeta_j = 2 \sum_{j=1}^m R_{ijij}(\zeta_i - \zeta_j) \qquad \text{for } \, i = 1,\dots,m
	$$
	hence
	\begin{equation} \label{GamZZ_zeta}
		\langle \Gamma Z,Z \rangle = \sum_{i=1}^m (\Gamma Z)_{ii} Z^{ii} = \sum_{i=1}^m (\Gamma Z)_{ii} \zeta_i = 2 \sum_{i,j=1}^m R_{ijij} \zeta_i(\zeta_i - \zeta_j)
	\end{equation}
	and this can be rewritten as
	$$
		\langle \Gamma Z,Z \rangle = \sum_{i,j=1}^m R_{ijij} \zeta_i(\zeta_i - \zeta_j) + \sum_{i,j=1}^m R_{jiji} \zeta_j(\zeta_j - \zeta_i) = \sum_{i,j=1}^m R_{ijij} (\zeta_i - \zeta_j)^2 \, .
	$$
	Since $\sum_{i=1}^m \zeta_i = 0$, we have $\sum_{i,j=1}^m \zeta_i \zeta_j = 0$ and therefore
	\begin{equation} \label{Z2}
		\sum_{i,j=1}^m (\zeta_i-\zeta_j)^2 = m\sum_{i=1}^m \zeta_i^2 + m\sum_{j=1}^m \zeta_j^2 - 2\sum_{i,j=1}^m \zeta_i\zeta_j = 2 m \sum_{i=1}^m \zeta_i^2 = 2m|Z|^2 \, .
	\end{equation}
	Moreover, for any $1 \leq k < t \leq m$ we have
	\begin{equation} \label{curv_term_Z}
		(\zeta_k - \zeta_t)^2 \leq 2(\zeta_k^2 + \zeta_t^2) \leq 2|Z|^2 = \frac{2}{m} \sum_{1\leq i<j\leq m} (\zeta_i-\zeta_j)^2 \, .
	\end{equation}
	We order the set $\{(i,j) \in \N\times\N : 1 \leq i<j \leq m\}$ as a sequence $\{(i_\alpha,j_\alpha)\}_{1\leq\alpha\leq\binom{m}{2}}$ so that
	$$
		R_{i_\alpha j_\alpha i_\alpha j_\alpha} \leq R_{i_\beta j_\beta i_\beta j_\beta} \qquad \forall \, 1 \leq \alpha \leq \beta \leq \binom{m}{2}
	$$
	and we set $\kappa_\alpha = R_{i_\alpha j_\alpha i_\alpha j_\alpha}$, $c_\alpha = (\zeta_{i_\alpha}-\zeta_{j_\alpha})^2$ for every $1\leq \alpha\leq\binom{m}{2}$. Then \eqref{curv_term_Z} reads as
	\begin{equation} \label{curv_term_Z2}
		c_\alpha \leq \frac{2}{m} \sum_{\beta=1}^{\binom{m}{2}} c_\beta \qquad \forall \, 1 \leq \alpha \leq \binom{m}{2}
	\end{equation}
	and \eqref{GamZZ_zeta}, \eqref{Z2} can be expressed as
	$$
		\langle \Gamma Z,Z \rangle = 2\sum_{1\leq i<j\leq m} R_{ijij} (\zeta_i - \zeta_j)^2 = 2 \sum_{\alpha=1}^{\binom{m}{2}} \kappa_\alpha c_\alpha \, , \qquad \sum_{\alpha=1}^{\binom{m}{2}} c_\alpha = m|Z|^2 \, .
	$$
	Applying Remark \ref{rem_ab} with $N = \binom{m}{2}$, $k = \frac{m}{2} < \binom{m}{2}$ and using \eqref{m2_sect} we conclude
	$$
		\langle \Gamma Z,Z \rangle \geq \frac{2}{\lfloor\frac{m}{2}\rfloor} \sum_{\alpha=1}^{\lfloor\frac{m}{2}\rfloor} \kappa_\alpha \sum_{\beta=1}^{\binom{m}{2}} c_\beta \geq 2 C \sum_{\beta=1}^{\binom{m}{2}} c_\beta = 2mC|Z|^2 \, .
	$$
\end{proof}

\begin{lemma}[\cite{tachi74}] \label{lem_That}
	Let $T$ be an algebraic curvature tensor. Then
	\begin{equation} \label{That_P}
		|\hat T|^2 = 2(m-1) |P|^2
	\end{equation}
	where $P$ is the tensor defined in \eqref{P_def} and $\hat T$ is defined as in \eqref{Qhat}. In particular, if $T = W$ is totally traceless then
	\begin{equation} \label{What_W}
		|\hat W|^2 = 2(m-1) |W|^2 \, .
	\end{equation}
\end{lemma}

\begin{proof}
	From the defining formula \eqref{Qhat} we have
	$$
		2 \hat T_{ijktsr} = T_{sjkt} g_{ir} + T_{iskt} g_{jr} + T_{ijst} g_{kr} + T_{ijks} g_{tr} - T_{rjkt} g_{is} - T_{irkt} g_{js} - T_{ijrt} g_{ks} - T_{ijkr} g_{ts} \, .
	$$
	A direct computation, using the symmetries of $T$, yields
	\begin{align*}
		\hat T^{ijktsr} \hat T_{ijktsr} & = \hat T^{ijktsr} ( T_{sjkt} g_{ir} + T_{iskt} g_{jr} + T_{ijst} g_{kr} + T_{ijks} g_{tr} ) = 4 \hat T^{ijktsr} T_{sjkt} g_{ir}
	\end{align*}
	and then
	\begin{align*}
		2 \hat T^{ijktsr} T_{sjkt} g_{ir} & = T^{sjkt} T_{sjkt} g^{ir} g_{ir} + T^{iskt} T_{sjkt} g^{jr} g_{ir} + T^{ijst} T_{sjkt} g^{kr} g_{ir} + T^{ijks} T_{sjkt} g^{tr} g_{ir} \\
		& \phantom{=\;} - T^{rjkt} T_{sjkt} g^{is} g_{ir} - T^{irkt} T_{sjkt} g^{js} g_{ir} - T^{ijrt} T_{sjkt} g^{ks} g_{ir} - T^{ijkr} T_{sjkt} g^{ts} g_{ir} \\
		& = m T^{sjkt} T_{sjkt} + T^{iskt} T_{sikt} + T^{ijst} T_{sjit} + T^{ijks} T_{sjki} - T^{rjkt} T_{rjkt} - 2 E^{jt} E_{jt} \\
		& = (m-1)|T|^2 - 2|E|^2 + T^{iskt}(T_{sikt} + T_{ksit} + T_{tski}) \\
		& = (m-1)|T|^2 - 2|E|^2 + T^{iskt}(T_{sikt} + T_{ksit} + T_{ikst}) \\
		& = (m-1)|T|^2 - 2|E|^2
	\end{align*}
	where in the last equality we have used the fact that $T$ satisfies the first Bianchi identity. The conclusion then follows by \eqref{|P|}.
\end{proof}

\begin{lemma}[\cite{pw20}] \label{lem_PW_1}
	Let $T$ be an algebraic curvature tensor and $\omega$ a $2$-form. Then
	\begin{equation} \label{PW_Tomega}
		\omega_{ij} \omega_{kt} \hat T_{abcd}^{\phantom{abcd}ij} \hat T^{abcdkt} \leq 4 \omega_{ij} \omega^{ij} T_{abcd} T^{abcd} \, .
	\end{equation}
\end{lemma}

\begin{proof}
	The values appearing on both sides \eqref{PW_Tomega} do not depend on the local coframe $\{\theta^i\}$ chosen to perform computations. For the sake of simplicity, we assume that $\{\theta^i\}$ is an orthonormal coframe. Note that in this case we can avoid raising and lowering indexes to denote contraction with $g$ or $g^{-1}$, since $g^{ij} = g_{ij} = \delta_{ij}$, the Kronecker symbol.
	
	First, we observe that
	\begin{equation} \label{Tomega}
		\omega_{ij} \widehat T_{abcdij} = \omega_{ia} T_{ibcd} + \omega_{ib} T_{aicd} + \omega_{ic} T_{abid} + \omega_{id} T_{abci} \, .
	\end{equation}
	Then, we assume that the coframe $\{\theta^i\}$ is chosen so that $\omega$ can be expressed as
	$$
		\omega = \omega_{12} \, \theta^1 \wedge \theta^2 + \omega_{34} \, \theta^3 \wedge \theta^4 + \cdots + \omega_{2k-1,2k} \, \theta^{2k-1} \wedge \theta^{2k}
	$$
	with $k = \lfloor\frac{m}{2}\rfloor$. For every $1 \leq a \leq m$, set
	$$
		a' = \begin{cases}
			a - 1 & \text{if } \, a \leq 2k, \, a \text{ even,} \\
			a + 1 & \text{if } \, a \leq 2k, \, a \text{ odd,} \\
			a & \text{otherwise.}
		\end{cases}
	$$
	Then, \eqref{Tomega} rewrites as (no summation is intended over repeated indexes on the RHS)
	\begin{equation} \label{Tomega1}
		\omega_{ij}\widehat T_{abcdij} = \omega_{a'a} T_{a'bcd} + \omega_{b'b}T_{ab'cd} + \omega_{c'c}T_{abc'd} + \omega_{d'd}T_{abcd'} \, .
	\end{equation}
	By Cauchy's inequality we can bound
	$$
		(\omega_{ij}\widehat T_{abcdij})^2 \leq \left( \omega_{a'a}^2 + \omega_{b'b}^2 + \omega_{c'c}^2 + \omega_{d'd}^2 \right) \left( T_{a'bcd}^2 + T_{ab'cd}^2 + T_{abc'd}^2 + T_{abcd'}^2 \right)
	$$
	but in fact we also have the more effective bound
	\begin{equation} \label{Tomega2}
		(\omega_{ij}\widehat T_{abcdij})^2 \leq \left( \sum_{i,j=1}^m \omega_{ij}^2 \right) \left( T_{a'bcd}^2 + T_{ab'cd}^2 + T_{abc'd}^2 + T_{abcd'}^2 \right) \, .
	\end{equation}
	Since for every $1\leq i,j\leq m$ we have $\omega_{ij}^2 = \omega_{ji}^2$, to justify deduction of \eqref{Tomega2} from \eqref{Tomega1} one observes that, up to dropping out vanishing terms from the RHS of \eqref{Tomega1}, for every $i\neq j$ there are at most two sets, amongst $\{a,a'\}$, $\{b,b'\}$, $\{c,c'\}$ and $\{d,d'\}$, that coincide with $\{i,j\}$. Indeed, if $a=b$ then $a'=b'$, and $\omega_{a'a} = \omega_{b'b}$ while $T_{a'bcd} = T_{b'acd} = - T_{ab'cd}$; if $a=b'$ then $b=a'$ and $T_{a'bcd} = 0 = T_{ab'cd}$. Hence, for any $a,b,c,d$ we have
	\begin{align*}
		\omega_{a'a} T_{a'bcd} + \omega_{b'b} T_{ab'cd} \neq 0 \qquad & \Rightarrow \qquad \{a,a'\}\cap\{b,b'\} = \emptyset \\
		\intertext{and similarly}
		\omega_{c'c} T_{abc'd} + \omega_{d'd} T_{abcd'} \neq 0 \qquad & \Rightarrow \qquad \{c,c'\}\cap\{d,d'\} = \emptyset \, .		
	\end{align*}
	Summing over all tuples $(a,b,c,d)$, we obtain
	$$
		\omega_{ij}\omega_{kt} \widehat T_{abcdij}\widehat T_{abcdkt} \leq \omega_{ij}\omega_{ij} \sum_{a,b,c,d=1}^m \left( T_{a'bcd}^2 + T_{ab'cd}^2 + T_{abc'd}^2 + T_{abcd'}^2 \right) = 4 \omega_{ij} \omega_{ij} T_{abcd} T_{abcd}
	$$
	where equality follows since the map $a\mapsto a'$ is a bijection of $\{1,\dots,m\}$ into itself, so that
	$$
		\sum_{a,b,c,d=1}^m T_{a'bcd}^2 = \sum_{a,b,c,d=1}^m T_{abcd}^2
	$$
	and similarly for the other terms.
\end{proof}

\begin{proposition}[\cite{pw20}] \label{m-1/2_W}
	Let $x \in M$, $C\in\R$ and assume that
	$$
		\mathfrak R^{(\lfloor\frac{m-1}{2}\rfloor)}(x) \geq C \, .
	$$
	Then for every totally traceless algebraic curvature tensor $W$ we have
	$$
		\langle \Gamma W,W \rangle \geq 2(m-1)C|W|^2 \qquad \text{at } \, x .
	$$
\end{proposition}

\begin{proof}
	Let $\{\omega^\alpha\}_\alpha$ be an orthonormal basis of $\wedge^2_x M$ consisting of eigenvectors of $\mathfrak R$ with corresponding eigenvalues $\{\lambda_\alpha\}_\alpha$. Then, with respect to any local coframe $\{\theta^i\}$ we have
	\begin{equation} \label{R_gg_omega}
		R_{ijkt} = \sum_\alpha \lambda_\alpha \omega^\alpha_{ij} \omega^\alpha_{kt} \, , \qquad g_{ik} g_{jt} - g_{it} g_{jk} = 2 \sum_\alpha \omega^\alpha_{ij}\omega^\alpha_{kt} \, .
	\end{equation}
	From \eqref{Gam_hat} then we have
	\begin{equation} \label{PW_Gam_WW}
		\langle \Gamma W,W \rangle = \langle \mathfrak R^{T^0_4 M} \hat W, \hat W \rangle = \sum_{\alpha=1}^{\binom{m}{2}} \lambda_\alpha \omega^\alpha_{ij}\omega^\alpha_{kt} \hat W_{abcd}^{\phantom{abcd}ij} \hat W^{abcdkt} = \sum_{\alpha=1}^{\binom{m}{2}} \lambda_\alpha c_\alpha
	\end{equation}
	where we have set $c_\alpha = \omega^\alpha_{ij}\omega^\alpha_{kt} \hat W_{abcd}^{\phantom{abcd}ij} \hat W^{abcdkt}$. By \eqref{What_W} and the second in \eqref{R_gg_omega} we have
	$$
		2(m-1)|W|^2 = |\hat W|^2 = g_{ij} g_{kt} \hat W_{abcd}^{\phantom{abcd}ij} \hat W^{abcdkt} = \frac{1}{2}(g_{ij} g_{kt} - g_{it}g_{jk}) \hat W_{abcd}^{\phantom{abcd}ij} \hat W^{abcdkt} = \sum_{\alpha=1}^{\binom{m}{2}} c_\alpha
	$$
	and then by \eqref{PW_Tomega} for every $\alpha$
	\begin{equation} \label{PW_|W|}
		c_\alpha \leq 4 |W|^2 = \frac{2}{m-1} |\hat W|^2 = \frac{2}{m-1} \sum_{\beta=1}^{\binom{m}{2}} c_\beta \qquad \forall \, 1\leq \alpha\leq\binom{m}{2} \, .
	\end{equation}
	Applying Remark \ref{rem_ab} we obtain the desired conclusion.
\end{proof}

\begin{theorem} \label{thm_m-1/2_R}
	Let $x \in M$, $C\in\R$ and assume that
	\begin{equation} \label{m-1/2_R}
		\mathfrak R^{(\lfloor\frac{m-1}{2}\rfloor)}(x) \geq C \, .
	\end{equation}
	Then for every algebraic curvature tensor $T$ we have
	\begin{equation}
		\langle \Gamma T,T \rangle \geq 2(m-1)C|P|^2 \qquad \text{at x,}
	\end{equation}
	where $P$ is the tensor defined in \eqref{P_def}.
\end{theorem}

\begin{proof}
	First, recall from \eqref{Gam_TT} that
	\begin{equation} \label{Gam_TT1}
		\langle \Gamma T,T \rangle = \langle \Gamma W,W \rangle + \frac{4}{m-2} \langle \Gamma Z,Z \rangle \, .
	\end{equation}
	By Proposition \ref{m-1/2_W} we have
	\begin{equation} \label{Gam_TT2}
		\langle \Gamma W,W \rangle \geq 2(m-1)C |W|^2 \, .
	\end{equation}
	By \eqref{R(h)(k)}, from \eqref{m-1/2_R} we deduce $\mathfrak R^{(\lfloor\frac{m}{2}\rfloor)}(x) \geq C$ and then by \eqref{Riem_K_bound}
	$$
		\frac{1}{\lfloor\frac{m}{2}\rfloor} \sum_{i=1}^{\lfloor\frac{m}{2}\rfloor} \Sect(\pi_i) \geq \frac{C}{2}
	$$
	for every set $\{\pi_1,\dots,\pi_{\lfloor\frac{m}{2}\rfloor}\}$ of mutually orthogonal $2$-planes in $T_x M$. Then, by Proposition \ref{ZRic}
	\begin{equation} \label{Gam_TT3}
		\langle \Gamma Z,Z \rangle \geq m C|Z|^2 \, .
	\end{equation}
	Putting together \eqref{Gam_TT1}, \eqref{Gam_TT2}, \eqref{Gam_TT3} and using \eqref{|P|} we conclude
	$$
		\langle \Gamma T,T \rangle \geq C \left[ 2(m-1) |W|^2 + \frac{4m}{m-2} |Z|^2 \right] = 2(m-1) C|P|^2 \, .
	$$
\end{proof}

The combination of Proposition \ref{Boch_T_prop1} and Theorem \ref{thm_m-1/2_R} yields Theorem \ref{Boch_T}.

\section{Tachibana-type theorems} \label{sec_Tachi}

\subsection{The compact case}

\begin{theorem} \label{cpt_Tachi}
	Let $(M,\metric)$ be a compact Riemannian manifold of dimension $m\geq 3$ satisfying $\mathfrak R^{(\lfloor\frac{m-1}{2}\rfloor)} \geq 0$. If $T$ is a harmonic algebraic curvature tensor on $M$, then $\nabla T\equiv 0$. Moreover, if $\mathfrak R^{(\lfloor\frac{m-1}{2}\rfloor)} > 0$ at some point then $T$ is a constant multiple of $\metric \owedge \metric$.
\end{theorem}

\begin{proof}
	By Theorem \ref{Boch_T}, we have
	$$
		\frac{1}{2}\Delta|T|^2 \geq |\nabla T|^2 + (m-1) \mathfrak R^{(\lfloor\frac{m-1}{2}\rfloor)} |P_T|^2 \geq 0
	$$
	where $P_T$ is the pseudo-projective tensor field associated to $T$ as in \eqref{P_def}. By compactness of $M$, the subharmonic function $|T|^2$ must be constant, hence $|\nabla T|^2 \equiv 0$ and $\mathfrak R^{(\lfloor\frac{m-1}{2}\rfloor)} |P_T|^2 \equiv 0$ since $\mathfrak R^{(\lfloor\frac{m-1}{2}\rfloor)} \geq 0$. In particular we have $\nabla T \equiv 0$ as claimed.
	
	By parallelism of the metric, $\nabla T \equiv 0$ implies $\nabla E_T \equiv 0$, where $E_T$ is the Ricci contraction of $T$, and then $\nabla P_T \equiv 0$. In particular, $|P_T|$ is constant. If $\mathfrak R^{(\lfloor\frac{m-1}{2}\rfloor)} > 0$ at some point then necessarily $|P_T|\equiv 0$ on $M$, that is, $P_T\equiv 0$ and this yields $T = c \metric \owedge \metric$, where $c = \frac{\trace_g E_T}{2m(m-1)} = \frac{S_T}{2m(m-1)}$ is also constant by Proposition \ref{B(T)_cons}.
\end{proof}

As a direct consequence we have the following generalization of Tachibana's theorem.

\begin{theorem} \label{cor_cpt_tachi}
	Let $M$ be a compact Riemannian manifold of dimension $m\geq 3$ with harmonic curvature and $\mathfrak R^{(\lfloor\frac{m-1}{2}\rfloor)} \geq 0$. Then $M$ is locally symmetric. Moreover, if $\mathfrak R^{(\lfloor\frac{m-1}{2}\rfloor)} > 0$ somewhere then $M$ is isometric to a quotient of $\sphere^m$.
\end{theorem}

In cases $m=3,4$ we have $\lfloor\frac{m-1}{2}\rfloor = 1$, hence the assumption on the curvature operator reduce to $\mathfrak R \geq 0$ (and possibly $\mathfrak R > 0$ somewhere, for the second part of the theorem) as in the standard Tachibana's theorem. However, we point out that in case $m=3$ the non-negativity (resp., positivity) of the curvature operator can be relaxed to the milder condition that the Ricci curvature is non-negative (resp., positive).

\begin{theorem} \label{thm_tachi_3}
	Let $M^3$ be a compact $3$-dimensional Riemannian manifold with harmonic curvature and $\Ricc \geq 0$. Then $M$ is isometric to a quotient of $\sphere^3$, $\sphere^2\times\R$ or $\R^3$. Moreover, if $\Ricc > 0$ somewhere then $M$ is isometric to a quotient of $\sphere^3$.
\end{theorem}

\begin{proof}
	Since $M$ has harmonic curvature, by Proposition \ref{Boch_T_prop1}
	\begin{equation} \label{Riem_3_Boch}
		\frac{1}{2} \Delta|\Riem|^2 = |\nabla\Riem|^2 + \frac{1}{2} \langle \Gamma\Riem,\Riem \rangle \, .
	\end{equation}
	As $\dim M=3$, the Weyl tensor vanishes. So, by Proposition \ref{Gam_TT_gen} and formula \eqref{Gam_EE0}
	\begin{equation} \label{Gam3}
		\frac{1}{2} \langle \Gamma\Riem,\Riem \rangle = \frac{2}{m-2} \langle \Gamma\Ricc,\Ricc \rangle = 2 \langle \Gamma\Ricc,\Ricc \rangle = 4 (R_{ij} R^j_{\;k} R^{ki} - R_{ijkt} R^{ik} R^{jt}) \, .
	\end{equation}
	Moreover, again since $W\equiv 0$, we have
	$$
		R_{ijkt} = R_{ik} g_{jt} + R_{jt} g_{ik} - R_{it} g_{jk} - R_{jk} g_{it} - \frac{S}{2} (g_{ik} g_{jt} - g_{it} g_{jk})
	$$
	and substituting this into \eqref{Gam3} we obtain
	\begin{align*}
		\frac{1}{8} \langle \Gamma\Riem,\Riem \rangle & = R_{ij} R^j_{\;k} R^{ki} - 2 S R_{ij} R^{ij} + 2 R_{it} R^{ik} R_k^{\;\,t} + \frac{S^3}{2} - \frac{S}{2} R^{ik} R_{ik} \\
		& = 3 R_{ij} R^j_{\;k} R^{ki} - \frac{5}{2} S R_{ij} R^{ij} + \frac{S^3}{2} \, .
	\end{align*}
	Denoting by $\lambda,\mu,\nu$ the eigenvalues of the Ricci operator, we have
	\begin{alignat*}{2}
		R_{ij} R^j_{\;k} R^{ki} & = \trace(\Ricc^3) && = \lambda^3 + \mu^3 + \nu^3 \\
		R_{ij} R^{ij} & = \trace(\Ricc^2) && = \lambda^2 + \mu^2 + \nu^2 \\
		S & = \trace(\Ricc) && = \lambda + \mu + \nu
	\end{alignat*}
	and therefore
	\begin{align*}
		\frac{1}{4} \langle \Gamma\Riem,\Riem \rangle & = 6 (\lambda^3 + \mu^3 + \nu^3) - 5 (\lambda + \mu + \nu) (\lambda^2 + \mu^2 + \nu^2) + (\lambda + \mu + \nu)^3 \\
		& = 6 (\lambda^3 + \mu^3 + \nu^3) - (\lambda + \mu + \nu) [ 5 (\lambda^2 + \mu^2 + \nu^2) - (\lambda + \mu + \nu)^2 ] \\
		& = 6 (\lambda^3 + \mu^3 + \nu^3) - (\lambda + \mu + \nu) [ 4 (\lambda^2 + \mu^2 + \nu^2) - 2 (\lambda\mu + \lambda\nu + \mu\nu) ] \\
		& = 6 (\lambda^3 + \mu^3 + \nu^3) \\
		& \phantom{=\;} - 4 (\lambda^3 + \mu^3 + \nu^3) - 4(\lambda\mu^2 + \lambda\nu^2 + \mu\lambda^2 + \mu\nu^2 + \nu\lambda^2 + \nu\mu^2) \\
		& \phantom{=\;} + 2(\lambda^2\mu + \lambda^2\nu + \lambda\mu^2 + \mu^2\nu + \lambda\nu^2 + \mu\nu^2) + 6 \lambda\mu\nu \\
		& = 2 [ \lambda^3 + \mu^3 + \nu^3 - \lambda\mu^2 - \lambda\nu^2 - \mu\lambda^2 - \mu\nu^2 - \nu\lambda^2 - \nu\mu^2 + 3 \lambda\mu\nu ] \\
		& = 2 [ \lambda(\lambda-\mu)(\lambda-\nu) + \mu(\mu-\lambda)(\mu-\nu) + \nu(\nu-\lambda)(\nu-\mu) ] \, .
	\end{align*}
	We show that this term is non-negative if $\Ricc\geq0$. Without loss of generality, we can assume $\lambda\leq\mu\leq\nu$. Then we can write
	$$
		\mu = t\nu + (1-t)\lambda \equiv \lambda + ts \equiv \nu - (1-t)s
	$$
	for some $0\leq t\leq 1$, where $s = \nu-\lambda\geq0$. Thus
	\begin{align*}
		\lambda(\lambda-\mu)(\lambda-\nu) & = \lambda( \lambda - \lambda - ts ) (-s) = \lambda t s^2 \\
		\mu(\mu-\lambda)(\mu-\nu) & = \mu (\lambda + ts - \lambda)(\nu - (1-t)s - \nu) = -\mu t(1-t)s^2 \\
		\nu(\nu-\lambda)(\nu-\mu) & = \nu s (\nu - \nu + (1-t)s) = \nu(1-t)s^2
	\end{align*}
	and then
	\begin{align*}
		\frac{1}{4} \langle \Gamma\Riem,\Riem \rangle & = 2 [ \lambda t - \mu t (1-t) + \nu (1-t) ] s^2 \\
		& = 2 [ \lambda t - (t\nu + (1-t)\lambda) t(1-t) + \nu (1-t) ] s^2 \\
		& = 2 [ \lambda t - t(1-t)^2 \lambda - t^2 (1-t) \nu + \nu (1-t) ] s^2 \\
		& = 2 [ t^2(2-t) \lambda + (1-t)^2(1+t) \nu ] s^2 \\
		& \geq 0
	\end{align*}
	where the inequality holds since $2-t\geq0$ and $1+t\geq0$ by construction, while $\lambda,\nu\geq 0$ since we have assumed $\Ricc\geq0$. Moreover, equality holds if and only if one the following cases occurs:
	\begin{itemize}
		\item [(i)] $s=0$, that is, $\lambda = \mu = \nu$;
		\item [(ii)] $\lambda=0$ and $t=1$, that is, $0 = \lambda \leq \mu = \nu$.
	\end{itemize}
	Having established $\langle \Gamma\Riem,\Riem \rangle \geq 0$, applying the divergence theorem (or the maximum principle) to \eqref{Riem_3_Boch} we see that $\nabla\Riem\equiv0$, that is, $M$ is locally symmetric. The Cotton tensor of $M$ is zero by harmonicity of the curvature, hence $M$ is locally conformally flat and by Noronha's Theorem \ref{noronha} we conclude that $M$ is isometric to a quotient of either $\sphere^3$, $\sphere^2\times\R$ or $\R^3$. Note that all three cases are compatible with either condition (i) or (ii) mentioned above. Lastly, if $\Ricc>0$ at some point then $M$ is necessarily a quotient of $\sphere^3$.
\end{proof}

If $m = \dim M \geq 4$ and $M$ is locally conformally flat, the conclusions of Tachibana's theorem also hold under the assumption that $\Sect^{(\lfloor\frac{m}{2}\rfloor)} \geq 0$, with possibly strict inequality at some point. 
Recall that $\Sect^{(k)}$ has been defined in Definition \ref{def_Scalk} and that condition $\Sect^{(\lfloor\frac{m}{2}\rfloor)} \geq c$ is weaker than $\mathfrak R^{(\lfloor\frac{m-1}{2}\rfloor)} \geq c$ by \eqref{Riem_K_bound}. 

\begin{proposition} \label{cor_W0_Sc_class}
	Let $(M^m,\metric)$ be a compact, locally conformally flat Riemannian manifold with constant scalar curvature $S$. If $\Sect^{(\lfloor\frac{m}{2}\rfloor)} \geq 0$ then $M$ is locally symmetric, and if $\Sect^{(\lfloor\frac{m}{2}\rfloor)} > 0$ at some point then $M$ is a quotient of $\sphere^m_S$.
\end{proposition}

\begin{proof}
	Since $W\equiv0$ and $S$ is constant, $M$ has harmonic curvature and thus
	$$
		\frac{1}{2} \Delta|\Riem|^2 = |\nabla\Riem|^2 + \frac{1}{2} \langle \Gamma\Riem, \Riem \rangle \, .
	$$
	As $W=0$ we have $\langle\Gamma\Riem,\Riem\rangle = \frac{4}{m-2} \langle \Gamma\Ricc,\Ricc \rangle$, so by Proposition \ref{ZRic} we estimate
	$$
		\frac{1}{2} \Delta|\Riem|^2 \geq |\nabla\Riem|^2 + \frac{8m}{m-2} \Sect^{(\lfloor\frac{m}{2}\rfloor)} |\overset{\circ}{\Ricc}|^2
	$$
	where $\overset{\circ}{\Ricc}$ is the traceless part of $\Ricc$. Then the desired conclusion follows reasoning as in the proof of Theorem \ref{cpt_Tachi}.
\end{proof}

\subsection{Bounded subharmonic functions on complete manifolds with $\Ricc\geq0$}

In this subsection we collect a series of results that will be useful in the following one to deal with the case of complete manifolds with harmonic curvature. We first have the following mean value inequalities for subharmonic functions due to P. Li, \cite{li86}, and Li-Schoen, \cite{ls84}.

\begin{proposition}[\cite{li86}, Theorem 4] \label{prop_li_mean}
	Let $(M,\metric)$ be a complete Riemannian manifold with $\Ricc\geq0$. Let $f\in L^\infty(M)$ be a subharmonic function. Then for any $x\in M$
	\begin{equation}
		\lim_{R\to+\infty} \frac{1}{|B_R(x)|}\int_{B_R(x)} f = \sup_M f \, .
	\end{equation}
\end{proposition}

\begin{proposition}[\cite{ls84}, Theorem 2.1] \label{prop_li_schoen}
	Let $(M^m,\metric)$ be a complete Riemannian manifold with $\Ricc\geq -(m-1)\kappa^2$. Let $R>0$, $x\in M$ and let $f\geq0$ be a subharmonic function defined on $B_R(x)$. There exists a constant $C=C(m,p)>0$ such that
	\begin{equation}
		\sup_{B_{(1-\tau)R}(x)} f^p \leq \tau^{-C(1+\kappa R)} \frac{1}{|B_R(x)|} \int_{B_R(x)} f^p
	\end{equation}
	for every $\tau \in (0,1/2)$.
\end{proposition}

\begin{corollary} \label{cor_li}
	Let $(M,\metric)$ be a complete Riemannian manifold with $\Ricc\geq0$. Let $f\geq0$ be a nonnegative subharmonic function. Then for any $p\in[1,+\infty)$
	\begin{equation} \label{li_p_mean}
		\lim_{R\to+\infty} \frac{1}{|B_R(x)|} \int_{B_R(x)} f^p = \sup_M f^p \, .
	\end{equation}
\end{corollary}

\begin{proof}
	First observe that if $f\geq0$ is subharmonic, then $f^p$ is also subharmonic for any $p\geq 1$. If $f$ is bounded then $f^p$ is also bounded and the conclusion follows by Proposition \ref{prop_li_mean}. If $f$ is unbounded then by Proposition \ref{prop_li_schoen} both sides of \eqref{li_p_mean} equal $+\infty$ and the conclusion follows.
\end{proof}

The next Proposition \ref{prop_ccmtype}, together with its proof, rephrases in general terms an observation contained in \cite{ccm95}. 

\begin{proposition}[\cite{ccm95}] \label{prop_ccmtype}
	Let $(M,\metric)$ be a complete Riemannian manifold with $\Ricc\geq0$. Let $f\in L^\infty(M)$ be a subharmonic function. Then for any $x\in M$
	\begin{equation} \label{ccm}
		\lim_{R\to+\infty} \frac{R^2}{|B_R(x)|} \int_{B_R(x)} \Delta f = 0 \, .
	\end{equation}
\end{proposition}

\begin{proof}
	Let $r$ denote the distance function from $x$. By the Laplacian comparison theorem we have
	\begin{equation} \label{comp_r}
		\Delta r^2 \leq 2m
	\end{equation}
	where $m = \dim M$. Define $h = \sup_M f - f \geq 0$. Green's identities give
	\begin{align*}
		\int_{B_R(x)} \left(1-\frac{r^2}{R^2}\right) \Delta h + \frac{1}{R^2} \int_{B_R(x)} h \Delta r^2 = \frac{1}{R^2} \int_{\partial B_R(x)} h \langle \nabla r^2,\nu \rangle \geq 0
	\end{align*}
	for almost every $R>0$. Since $h\geq0$ and $-\Delta h = \Delta f \geq 0$, by \eqref{comp_r} we estimate
	$$
	\frac{2m}{R^2} \int_{B_R(x)} h \geq \frac{1}{R^2} \int_{B_R(x)} h \Delta r^2 \geq \int_{B_R(x)} \left(1-\frac{r^2}{R^2}\right) \Delta f \geq \frac{3}{4} \int_{B_{R/2}(x)} \Delta f \, .
	$$
	Hence, dividing by $|B_{R/2}(x)|$ we have
	\begin{equation} \label{ccm1}
		\frac{1}{|B_{R/2}(x)|} \int_{B_R(x)} h \geq \frac{3}{8m} \frac{R^2}{|B_{R/2}(x)|} \int_{B_{R/2}(x)} \Delta f \geq 0 \, .
	\end{equation}
	By Bishop-Gromov theorem we also have $|B_{R/2}(x)| \geq 2^m |B_R(x)|$, hence
	\begin{equation} \label{ccm2}
		\frac{2^m}{|B_R(x)|} \int_{B_R(x)} h \geq \frac{1}{|B_{R/2}(x)|} \int_{B_R(x)} h
	\end{equation}
	and by Proposition \ref{prop_li_mean}
	\begin{equation} \label{ccm3}
		\frac{1}{|B_R(x)|} \int_{B_R(x)} h = \sup_M f - \frac{1}{|B_R(x)|} \int_{B_R(x)} f \to 0 \qquad \text{as } \, R \to +\infty \, .
	\end{equation}
	Putting together \eqref{ccm1}, \eqref{ccm2} and \eqref{ccm3} we obtain \eqref{ccm}.
\end{proof}

Let $(M,\metric)$ be a complete Riemannian manifold. For every $x\in M$, $R>0$ and for every measurable function $\psi$ on $B_R(x)$ we define
$$
	\psi_{x,R} = \frac{1}{|B_R(x)|} \int_{B_R(x)} \psi
$$
whenever the RHS of this inequality happens to be well defined. From the work of P. Buser, \cite{bus82}, combined with Cheeger's inequality it is known (see for instance L. Saloff-Coste, \cite[page 439]{sc92}) that geodesic balls of a complete Riemannian manifold with non-negative Ricci curvature support the following Poincaré inequality.

\begin{proposition}[\cite{bus82},\cite{sc92}] \label{prop_poinc}
	Let $(M^m,\metric)$ be a complete Riemannian manifold with $\Ricc\geq0$. Then, there exists $C=C(m)>0$ such that for every $x\in M$ and $R>0$
	\begin{equation}
		\int_{B_R(x)} |f-f_{x,R}|^2 \leq C R^2 \int_{B_R(x)}|\nabla f|^2 \qquad \forall \, f \in C^\infty(B_R(x)) \, .
	\end{equation}
\end{proposition}

We are now in the position to prove the next Liouville-type theorem.

\begin{theorem} \label{pr_liou}
	Let $(M,\metric)$ be a complete, noncompact Riemannian manifold with $\Ricc\geq0$. Let $a\geq 0$ be a measurable function on $M$ and let $0 \leq f \in L^\infty(M)$ satisfy
	$$
	\Delta f \geq a f \qquad \text{on } \, M \, .
	$$
	Assume that for some $x\in M$ one of the following conditions is satisfied:
	\begin{itemize}
		\item [i)] $f_{x,R} \to 0$ as $R\to+\infty$,
		\item [ii)] for some constant $C_0>0$ and for some compact set $K\subsetneq M$
		\begin{equation} \label{aC0}
			a \geq \frac{C_0^2}{r^2} \qquad \text{on } \, M \setminus K
		\end{equation}
		where $r$ is the distance function from $x$,
		\item [iii)] for some constants $C_1,C_2>0$
		\begin{equation} \label{aC12}
			a_{x,R} \geq \frac{C_1}{R^2} \, , \qquad \frac{1}{|B_R(x)|} \int_{B_R(x)} |a-a_{x,R}|^2 \leq \frac{C_2}{R^4}
		\end{equation}
		for all sufficiently large $R>0$.
	\end{itemize}
	Then $f\equiv 0$.
\end{theorem}

\begin{proof}
	Since $a\geq0$ and $f\geq0$, we have that $f$ is a bounded, nonnegative subharmonic function on $M$. Hence, by Proposition \ref{prop_li_mean} the limit
	$$
	\ell = \lim_{R\to+\infty} f_{x,R}
	$$
	exists and equals $\sup_M f \in [0,+\infty)$, and by Proposition \ref{prop_ccmtype} we also have
	\begin{equation} \label{liou_af}
		\lim_{R\to+\infty} R^2 (af)_{x,R} = 0 \, .
	\end{equation}
	Since $f\geq0$, the desired conclusion $f\equiv 0$ is equivalent to having $\ell=0$. Note that this is, in turn, equivalent to i). Hence let us assume, by contradiction, that $\ell>0$. Then we must be in either case ii) or iii). In both cases we aim at showing that \eqref{liou_af} cannot hold, hence concluding the proof by contradiction. 
	
	If ii) is in force, then fix $R_0>0$ large enough so that $K\subsetneq B_{R_0}(x)$. For every $R>R_0$ we have, using \eqref{aC0},
	$$
		R^2(af)_{x,R} \geq \frac{R^2}{|B_R(x)|} \int_{B_R(x)\setminus K} af \geq \frac{C_0}{|B_R(x)|} \int_{B_R(x)\setminus K} f = C_0 f_{x,R} - \frac{C_0}{|B_R(x)|} \int_K f \, .
	$$
	$M$ has infinite volume as it is a complete noncompact manifold with $\Ricc\geq0$, see for instance \cite[page 25]{sy94}, hence letting $R\to+\infty$ in the above inequality we obtain
	$$
	\liminf_{R\to+\infty} R^2(af)_{x,R} \geq C_0 \ell > 0 \, ,
	$$
	contradicting \eqref{liou_af}.
	
	If iii) is in force, then writing $a = a_{x,R} + (a-a_{x,R})$ and $f = f_{x,R} + (f-f_{x,R})$ one has
	$$
	(af)_{x,R} = a_{x,R}f_{x,R} + \frac{1}{|B_R(x)|} \int_{B_R(x)} (a-a_{x,R})(f-f_{x,R})
	$$
	for every $R>0$. Using Cauchy-Schwarz inequality together with \eqref{aC12} we further estimate
	\begin{equation} \label{liou_1}
		R^2(af)_{x,R} \geq C_1 f_{x,R} - \sqrt{C_2}\left( \frac{1}{|B_R(x)|} \int_{B_R(x)} |f-f_{x,R}|^2 \right)^{1/2} \, .
	\end{equation}
	The function $f^2$ is also bounded and subharmonic. In particular,
	$$
	\Delta f^2 = 2f\Delta f + 2|\nabla f|^2 \geq 2|\nabla f|^2
	$$
	and by Proposition \ref{prop_ccmtype} we get
	$$
	\lim_{R\to+\infty} \frac{R^2}{|B_R|} \int_{B_R} |\nabla f|^2 = 0 \, .
	$$
	Hence, by proposition \eqref{prop_poinc}
	\begin{equation} \label{liou_2}
		\lim_{R\to+\infty} \frac{1}{|B_R(x)|} \int_{B_R(x)} |f-f_{x,R}|^2 = 0
	\end{equation}
	and by \eqref{liou_1} and \eqref{liou_2} we obtain
	$$
	\liminf_{R\to+\infty} R^2(af)_{x,R} \geq C_1 \ell > 0 \, ,
	$$
	again contradicting \eqref{liou_af}.
	
\end{proof}

We remark that if $M$ is a complete \emph{parabolic} Riemannian manifold, in the sense of the subsequent Definition \ref{def_para}, then the analogue of Theorem \ref{pr_liou} holds with less restrictive conditions on $a$ and $f$ and no requirements on the Ricci tensor.

\begin{definition} \label{def_para}
	We say that a complete Riemannian manifold $M$ is parabolic if every upper bounded subharmonic function on $M$ is constant.
\end{definition}


This terminology originates from the complex analytic classification of (noncompact) Riemann surfaces, where the function theoretic property expressed by Definition \ref{def_para} distinguishes the parabolic from the hyperbolic ones, see \cite[Section IV.1.6]{as60}. 


For $M$ a complete Riemannian manifold of any dimension, $M$ is parabolic in the sense of Definition \ref{def_para} if and only if it does not admit any positive Green's function. A sufficient condition for parabolicity, which is also necessary for manifolds with non-negative Ricci curvature, see \cite{litam95,var81}, is that
\begin{equation} \label{para_cond}
	\limsup_{R\to+\infty} \int_1^R \frac{t}{|B_t|} \, \di t = +\infty
\end{equation}
where $|B_t|$ is the volume of the geodesic ball $B_t$ of radius $t$ centered at a fixed point $o\in M$.

\begin{theorem} \label{para_liou}
	Let $(M,\metric)$ be a complete, parabolic Riemannian manifold. Let $a\geq 0$ be a measurable function on $M$ and let $0 \leq f \in L^\infty(M)$ satisfy
	\begin{equation} \label{Dfaf}
		\Delta f \geq a f \qquad \text{on } \, M.
	\end{equation}
	Then $f$ is constant. Moreover, if $a>0$ somewhere then $f\equiv 0$.
\end{theorem}

\begin{proof}
	The function $f$ is bounded and subharmonic, hence it is constant by parabolicity of $M$ and from \eqref{Dfaf} it follows that $af\equiv 0$. If there exists $x\in M$ such that $a(x)>0$, then $f(x) = 0$ and thus $f\equiv 0$ on $M$.
\end{proof}

\subsection{The complete case}

In this subsection we prove Theorems \ref{int_cplt_t3}, \ref{int_para_tach1}, \ref{int_cpl_tach1} and \ref{int_cpl_tach2} from the Introduction. The following Proposition \ref{W0_Sc_class} is instrumental to the proof of all of them. It strengthens the thesis of Carron-Herzlich's classification Theorem \ref{carherz} for complete, locally conformally flat manifolds with $\Ricc\geq0$ under the additional assumption of constant scalar curvature.

\begin{proposition} \label{W0_Sc_class}
	Let $(M,g)$ be a complete and locally conformally flat Riemannian manifold of dimension $m\geq 3$ with $\Ricc \geq 0$ and constant scalar curvature $S$.
	\begin{itemize}
		\item[i)] If $S=0$ then $M$ is flat.
		\item[ii)] If $S>0$ then $M$ is either isometric to a quotient of $\R\times\sphere^{m-1}_{S/(m-1)(m-2)}$ or conformally equivalent to a quotient of $\sphere^m$.
	\end{itemize}
	In particular, if $\Ricc>0$ at some point then $M$ is conformally equivalent to a quotient of $\sphere^m$.
\end{proposition}

\begin{proof}
	i) If $S=0$ then $\Ricc\equiv 0$, and since we also have $W\equiv0$ we conclude that $\Riem\equiv0$.
	
	ii) If $S>0$, by the work of Zhu \cite{zhu94} and Carron, Herzlich \cite{ch06}, see Theorem \ref{carherz} from the Introduction, we know that the universal cover $(\tilde M,\tilde g)$ of $(M,g)$ satisfies one of the following:
	\begin{itemize}
		\item [a)] $(\tilde M,\tilde g)$ is isometric to $\R\times\sphere^{m-1}_{S/(m-1)(m-2)}$,
		\item [b)] $(\tilde M,\tilde g)$ is conformally equivalent to $\sphere^m$,
		\item [c)] $(\tilde M,\tilde g)$ is conformally equivalent to $\R^m$.
	\end{itemize}
	We repeat the argument of Theorem 1.1 of \cite{prs07} to show that c) cannot occur. Suppose, by contradiction, that c) holds. Then $\tilde M = \R^m$ and $\tilde g = u^{\frac{4}{m-2}}g_{\R^m}$ for some $0<u\in C^\infty(\R^m)$ satisfying the Yamabe equation
	\begin{equation} \label{yameq}
		c_m \Delta_{g_{\R^m}} u + S u^{\frac{m+2}{m-2}} = 0 \qquad \text{on } \, \R^m \, ,
	\end{equation}
	where $g_{\R^m}$ is the canonical Euclidean metric on $\R^m$. Since $S$ is a positive constant, by the celebrated work of Caffarelli, Gidas and Spruck \cite[Corollary 8.2]{cgs89} it follows that $u$ is radially symmetric around some point $x_0\in\R^m$ and has the expression
	$$
	u(x) = A(B+|x-x_0|^2)^{-\frac{m-2}{2}}
	$$
	for some positive constants $A,B$ only depending on $m$ and $S$. In particular, $(\tilde M,\tilde g)$ is an $m$-sphere of constant curvature with one point removed, hence it is not complete. But $(\tilde M,\tilde g)$ is the universal Riemannian cover of the complete manifold $(M,g)$, contradiction.
\end{proof}

We first deal with the $3$-dimensional case.

\begin{theorem} \label{cpl_tachi_dim3}
	Let $M^3$ be a complete Riemannian manifold of dimension $3$ with harmonic curvature and $\Ricc\geq0$. Then $M$ is isometric to a quotient of $\R^3$, $\sphere^2\times\R$ or $\sphere^3$. If $\Ricc>0$ at some point, then $M$ is isometric to a quotient of $\sphere^3$.
\end{theorem}

\begin{proof}
	The Ricci tensor of $M$ is Codazzi by the second Bianchi identity. In particular $M$ has constant scalar curvature and vanishing Cotton tensor (this is a particular case of Proposition \ref{B(T)_cons1}, for $T=\Riem$). Since $\dim M = 3$, the latter means that $M$ is locally conformally flat, see \cite[page 92]{eis49}. Then we can we can apply Proposition \ref{W0_Sc_class} to infer that either $M$ is isometric to a quotient of $\R^3$ or $\sphere^2\times\R$, or it is globally conformal to a quotient of $\sphere^3$. In the third scenario $M$ is compact, thus Theorem \ref{thm_tachi_3} applies and $M$ is in fact isometric to a quotient of $\sphere^3$. So the first part of the proposition is proved, and if $\Ricc>0$ at some point then the only possible conclusion is that $M$ is isometric to a quotient of $\sphere^3$.
\end{proof}


We now turn to the case $\dim M \geq 4$.

\begin{theorem} \label{cplt_tachi_1}
	Let $M$ be a complete Riemannian manifold of dimension $m\geq 4$ with harmonic curvature and $\mathfrak R^{(\lfloor\frac{m-1}{2}\rfloor)} \geq 0$. If the Weyl tensor satisfies
	$$
		\lim_{R\to+\infty} \frac{1}{|B_R(x)|} \int_{B_R(x)} |W|^p = 0
	$$
	for some $p\in[1,+\infty)$, then $M$ is isometric to a quotient of either $\sphere^m$, $\sphere^{m-1}\times\R$ or $\R^m$. Moreover, if $\mathfrak R^{(\lfloor\frac{m-1}{2}\rfloor)} > 0$ somewhere then $M$ is isometric to a quotient of $\sphere^m$.
\end{theorem}

\begin{proof}
	Since $M$ has harmonic curvature, the Weyl tensor is also harmonic and by Theorem \ref{Boch_T} we have
	\begin{equation} \label{divW1}
		\frac{1}{2}\Delta|W|^2 \geq |\nabla W|^2 \qquad \text{on } \, M.
	\end{equation}
	At any point where $|W|\neq 0$ we have
	$$
		\frac{1}{2}\Delta|W|^2 = \div(|W|\nabla|W|) = |W|\Delta|W| + |\nabla|W||^2 \qquad \text{and} \qquad |\nabla|W||^2 \leq |\nabla W|^2
	$$
	hence
	$$
		|W|\Delta|W| \geq |\nabla W|^2 - |\nabla|W||^2 \geq 0
	$$
	that is,
	$$
		\Delta|W| \geq 0 \, .
	$$
	Since $|W|\geq 0$ on $M$, any point where $|W|=0$ is a global mininum point for $|W|$, hence $\Delta|W|\geq0$ holds in the weak sense on the whole $M$. By Proposition \ref{prop_li_schoen} we deduce that $|W|\equiv 0$, hence $M$ is locally conformally flat. Also, $M$ has constant scalar curvature and nonnegative Ricci curvature, hence by Proposition \ref{W0_Sc_class} we have that one of the following cases occurs:
	\begin{itemize}
		\item [a)] $M$ is a quotient of $\R^m$,
		\item [b)] $M$ is a quotient of $\sphere^{m-1}\times\R$,
		\item [c)] $M$ is conformally equivalent to a quotient of $\sphere^m$.
	\end{itemize}
	If c) is in force, then $M$ is necessarily compact, so we can apply Theorem \ref{cor_cpt_tachi} to deduce that $M$ is locally symmetric. Since we also know that $M$ is locally conformally flat and we have $\Ricc\geq0$ as a consequence of $\mathfrak R^{(\lfloor\frac{m-1}{2}\rfloor)} \geq 0$, by Noronha's Theorem \ref{noronha} we conclude that $M$ is in fact \emph{isometric} to a quotient of $\sphere^m$, and this proves the first part of the thesis. If $\mathfrak R^{(\lfloor\frac{m-1}{2}\rfloor)} > 0$ at some point $x\in M$, then we also have $\Ricc > 0$ at $x$, so alternatives a) and b) are ruled out and the only possibility is that $M$ is a quotient of $\sphere^m$.
%
\end{proof}

\begin{theorem} \label{cplt_tachi_2}
	Let $M$ be a complete Riemannian manifold of dimension $m\geq 4$ with harmonic curvature. Assume that $\mathfrak R^{(\lfloor\frac{m-1}{2}\rfloor)} \geq a$ for some measurable function $a\geq0$ satisfying either condition ii) or iii) of Theorem \ref{pr_liou}. Then $M$ is compact, and in particular it is a quotient of $\sphere^m$.
\end{theorem}

\begin{proof}
	By the assumptions on $a$, we have that $\mathfrak R$ is $\lfloor\frac{m-1}{2}\rfloor$-nonnegative on $M$, and $\lfloor\frac{m-1}{2}\rfloor$-positive at some point. If $M$ is compact, then the conclusion follows by Corollary \ref{cor_cpt_tachi}. Hence, let us suppose (by contradiction) that $M$ is noncompact. Since $|W| \leq |\Riem|$ and $M$ has constant scalar curvature, by Corollary \ref{S_Riem_bound} we see that $|W|$ is bounded on $M$. Arguing as in the proof of Theorem \ref{cplt_tachi_1} we see that $|W|$ satisfies
	$$
		\Delta|W| \geq (m-1)a|W| \qquad \text{on } \, M
	$$
	(note that for $T=W$ the tensor $P$ defined as in \eqref{P_def} is $W$ itself), so by Theorem \ref{pr_liou} we deduce $|W|\equiv 0$, that is, $M$ is locally conformally flat. From this point on, we proceed as in the proof of Theorem \ref{cplt_tachi_1} and, since the assumptions on $\mathfrak R$ imply that $\Ricc > 0$ somewhere, the only possible conclusion is that $M$ is a quotient of a sphere, that in fact contradicts the noncompactness assumption. This concludes the proof.
\end{proof}

\begin{theorem} \label{para_cplt_tachi}
	Let $M$ be a complete Riemannian manifold of dimension $m\geq 4$ with harmonic curvature and $\mathfrak R^{(\lfloor\frac{m-1}{2}\rfloor)} \geq 0$. Assume that for some fixed origin $o\in M$
	\begin{equation} \label{para_cond_tachi}
		\lim_{R\to+\infty} \int_1^R \frac{t}{|B_t|} \, \di t = +\infty \, .
	\end{equation}
	Then $M$ is locally symmetric. If $\mathfrak R^{(\lfloor\frac{m-1}{2}\rfloor)}>0$ somewhere, then $M$ is a quotient of $\sphere^m$.
\end{theorem}

\begin{proof}
	As in the proof of Theorem \ref{cplt_tachi_2} we observe that $|W|$ is bounded and satisfies
	$$
		\Delta|W| \geq (m-1) \mathfrak R^{(\lfloor\frac{m-1}{2}\rfloor)} |W| \qquad \text{on } \, M.
	$$
	As remarked at the end of the previous subsection, condition \eqref{para_cond_tachi} implies that $M$ is parabolic in the sense of Definition \ref{def_para}. From parabolicity of $M$ and the assumption $\mathfrak R^{(\lfloor\frac{m-1}{2}\rfloor)} \geq 0$ we have that $|W|$ is constant on $M$. Then, by the Bochner inequality
	$$
		\frac{1}{2}\Delta|W|^2 \geq |\nabla W|^2
	$$
	it follows that $\nabla W \equiv 0$ and so, by Theorem 2 in Derdzi\'nski and Roter's paper \cite{dr77}, $M$ is either locally conformally flat or locally symmetric. But if $M$ is locally conformally flat, then we can argue as in the last part of the proof of Theorem \ref{cplt_tachi_1} to conclude that $M$ is isometric to a quotient of either $\R^m$, $\sphere^{m-1}\times\R$ or $\sphere^m$. Hence, in any case $M$ is locally symmetric. Lastly, if for some $x\in M$ we have $\mathfrak R^{(\lfloor\frac{m-1}{2}\rfloor)}(x)>0$ then the constant function $|W|$ must vanish, so $M$ is locally conformally flat. As just observed, in this case $M$ is isometric to a quotient of either $\R^m$, $\sphere^{m-1}\times\R$ or $\sphere^m$, but since also $\Ricc > 0$ we conclude that only the last possibility can occur.
\end{proof}

\subsection{The complete case: general curvature tensors}

\begin{theorem} \label{gen_cpl_T1}
	Let $M$ be a complete Riemannian manifold of dimension $m\geq 3$ satisfying $\mathfrak R^{(\lfloor\frac{m-1}{2}\rfloor)} \geq 0$. If $T$ is a harmonic algebraic curvature tensor on $M$ such that
	$$
		\lim_{R\to+\infty} \frac{1}{|B_R(x)|} \int_{B_R(x)} |W_T|^p + |Z_T|^p = 0
	$$
	for some $x\in M$ and $p\in[1,+\infty)$, where $W_T$ and $Z_T$ are the Weyl part of $T$ and the traceless part of the Ricci contraction of $T$. Then $T$ is a constant multiple of $\metric\owedge\metric$.
\end{theorem}

\begin{proof}
	Let $T = W_T + V_T + U_T$ be the orthogonal decomposition of $T$ given by \eqref{T_dec}-\eqref{VU_def}. From Proposition \ref{B(T)_cons} we see that the total trace $S_T$ of $T$ is constant, hence $U_T = \frac{S_T}{2m(m-1)} \metric \owedge \metric$ is parallel. By linearity, the tensor field $T' = W_T + V_T \equiv T - U_T$ is again a harmonic algebraic curvature tensor and its standard orthogonal decomposition $T' = W_{T'} + V_{T'} + U_{T'}$ is given by $W_{T'} = W_T$, $V_{T'} = V_T$, $U_{T'} = 0$. In particular, the traceless part $Z_{T'}$ of the Ricci contraction of $T'$ coincides with the analogous tensor $Z_T$ associated to $T$. By Theorem \ref{Boch_T} we have
	$$
		\frac{1}{2} \Delta|T'|^2 \geq |\nabla T'|^2 + (m-1)\mathfrak R^{(\lfloor\frac{m-1}{2}\rfloor)} |P_{T'}|^2 \geq 0
	$$
	where $P_{T'}$ is the pseudo-projective curvature tensor associated to $T'$ according to \eqref{P_def}, which coincides with the one associated to $T$. Arguing as in the proof of Theorem \ref{cplt_tachi_1} we see that $|T'|$ is a subharmonic function on $M$, and in particular that
	\begin{equation} \label{T'_Boch}
		\Delta|T'| \geq (m-1)\mathfrak R^{(\lfloor\frac{m-1}{2}\rfloor)} |P_{T'}|
	\end{equation}
	pointwise on $\{|T'| > 0\}$ and in the weak sense on $M$. Note that
	\begin{equation} \label{T'_norm}
		|T'|^2 = |W_{T'}|^2 + |V_{T'}|^2 = |W_T|^2 + |V_T|^2 = |W_T|^2 + \frac{4}{m-2} |Z_T|^2 \leq \left( |W_T| + \frac{2}{\sqrt{m-2}} |Z_T| \right)^2
	\end{equation}
	thus, since $(a+b)^p \leq 2^{p-1}(a^p + b^p)$ for any $a,b\geq0$ and $p\geq 1$,
	$$
		|T'|^p \leq 2^{p-1} |W_T|^p + \frac{2^p}{(m-2)^{(p-1)/2}} |Z_T|^p \, .
	$$
	In particular, under the assumptions of the present theorem we have
	$$
		\lim_{R\to+\infty} \frac{1}{|B_R(x)|} \int_{B_R(x)} |T'|^p = 0
	$$
	and by Corollary \ref{cor_li} we get $T'\equiv 0$ on $M$, that is, $T \equiv U_T$.
\end{proof}

\begin{theorem} \label{gen_cpl_T2}
	Let $M$ be a complete Riemannian manifold of dimension $m\geq 3$. Assume that $\mathfrak R^{(\lfloor\frac{m-1}{2}\rfloor)} \geq 0$ and that either
	\begin{itemize}
		\item [(a)] $\mathfrak R^{(\lfloor\frac{m-1}{2}\rfloor)} > 0$ somewhere on $M$ and \eqref{para_cond_tachi} is satisfied for some $o\in M$, or
		\item [(b)] $\mathfrak R^{(\lfloor\frac{m-1}{2}\rfloor)} \geq a$ for some measurable function $a\geq0$ satisfying ii) or iii) in Theorem \ref{pr_liou}.
	\end{itemize}
	If $T$ is a harmonic algebraic curvature tensor on $M$ such that
	\begin{equation} \label{Rsup_WZ}
		\limsup_{R\to+\infty} \frac{1}{|B_R(x)|} \int_{B_R(x)} |W_T|^p + |Z_T|^p < +\infty
	\end{equation}
	for some $x\in M$ and $p\in[1,+\infty)$, then $T$ is a constant multiple of $\metric\owedge\metric$.
\end{theorem}

\begin{remark} \label{rem_|T'|}
	Under the assumptions of Theorem \ref{gen_cpl_T2} it is easy to check (using for instance Propositions \ref{B(T)_cons1} and \ref{B(T)_cons}) that the algebraic curvature tensor fields $W_T$ and $Z_T\owedge\metric$ are both harmonic, hence $|W_T|$ and $|Z_T|$ are subharmonic functions. So, the $\limsup$ in \eqref{Rsup_WZ} is in fact a limit and in particular
	$$
		\lim_{R\to+\infty} \frac{1}{|B_R(x)|} \int_{B_R(x)} |W_T|^p + |Z_T|^p = \sup_M |W_T|^p + \sup_M |Z_T|^p \, .
	$$
	Hence, in this setting \eqref{Rsup_WZ} is equivalent to boundedness of the non-scalar part $W_T+V_T = W_T+\frac{4}{m-2} Z_T\owedge\metric$ of $T$.
\end{remark}

\begin{proof}[Proof of Theorem \ref{gen_cpl_T2}]
	Letting $T'$ and $P_{T'}$ be as in the proof of Theorem \ref{gen_cpl_T1}, by \eqref{|P|} and \eqref{T'_norm} we compute
	$$
		|P_{T'}|^2 = |W_T|^2 + \frac{2m}{(m-2)(m-1)}|Z_T|^2 \geq \frac{1}{2} |T'|^2
	$$
	and therefore from \eqref{T'_Boch} we get
	$$
		\Delta|T'| \geq \frac{(m-1)}{\sqrt{2}} a |T'| \, .
	$$
	By the previous Remark \ref{rem_|T'|} we see that $|T'|$ is a bounded function, so we infer $|T'|\equiv 0$ applying Theorem \ref{para_liou} or Theorem \ref{pr_liou}, depending on which assumption among (a) and (b) is in force. This shows that $T$ is a scalar multiple of $\metric\owedge\metric$, and the conclusion follows since the total trace $S_T$ of $T$ is constant.
\end{proof}

\bibliographystyle{plain}

\end{document}